\documentclass{amsart}

\usepackage{setting}

\numberwithin{equation}{section}

\begin{document}

\title{The Cohen-Macaulay Type of Edge-Weighted $r$-path Ideals} 


\author{Shuai Wei} 
\address{University of New Mexico
Department of Mathematics and Statistics, 
1 University of New Mexico, MSC01 1115
Albuquerque, NM 87131 USA}
\email{wei6@unm.edu}

\begin{abstract}
    We describe combinatorially the Cohen-Macaulay type of edge-weighted $r$-path suspensions of edge-weighted graphs for an arbitrary positive integer $r$. The computation of the Cohen-Macaulay type of edge-weighted suspensions of edge-weighted graphs becomes a special case of $r = 1$. 
\end{abstract}

\maketitle


\section*{Introduction} \label{1}

\begin{assumption*}
    Throughout, let $G$ be a (finite simple) graph with vertex set $V=V(G)=\{v_1,\ldots,v_d\}$ of cardinality $d \geq 1$ and edge set $E = E(G)$. An edge between vertices $v_i$ and $v_j$ is denoted $v_iv_j$. Let $\bbK$ be a field and set $R=\bbK[X_1,\ldots,X_d]$. Set $\ffm = (X_1,\ldots,X_d)R$. Fix an integer $r \in \bbN = \{1,2,\dots\}$ and set $R' = \bbK[\{X_{i,j} \mid i = 1,\ldots,d,j = 0,\ldots,r\}]$. An \emph{edge-weighting} on $G$ is a function $\omega: E \to \bbN$, and $G_\omega$ denotes a graph $G$ equipped with an edge-weighting $\omega$.
\end{assumption*}

    Combinatorial commutative algebra uses combinatorics and graph theory to understand certain algebraic constructions; it also uses algebra to understand certain objects in combinatorics and graph theory. In this paper, we explore aspects of this area via edge ideals and path ideals of edge-weighted graphs. 

    The \emph{edge ideal} of $G$ introduced by Villarreal~\cite{MR1031197} is the ideal $I(G)$ of $R$ that is ``generated by the edges of $G$'':
    \[I(G)=(X_iX_j\mid v_iv_j\in E)R.\]

    Villarreal~\cite{MR1031197} characterizes the trees $T$ for which $I(T)$ is Cohen-Macaulay: these are the ``suspensions'' or ``whiskered trees'', i.e., trees obtained from a subtree $U$ by adding an edge $\begin{tikzcd}v_i \ar[r,dash] &[-10pt]  v_{i1} \end{tikzcd}$ to each vertex $v$ of $U$: 
    \[
        \begin{tikzcd}
            U = &[-10pt] v_1 \ar[r,dash] & v_2 \ar[r,dash] & v_3 &[+20pt] T = &[-10pt] v_1 \ar[r,dash] \ar[d,dash] & v_2 \ar[r,dash] \ar[d,dash] & v_3 \ar[d,dash] \\
            & & & & & v_{11} & v_{2,1} & v_{3,1}
        \end{tikzcd}
    \]
    It is straightforward to show that the elements $v_i-v_{i1}$ form a maximal regular sequence on $R'/I(T)$ such that the ensuing quotient is $R/(I(U) + \langle x_1^2,\dots,x_d^2 \rangle)$. From this, one readily computes the Cohen-Macaulay type of $R'/I(T)$ as the number of ideals in an irredundant irreducible decomposition of $I(U)$, in other words, the number of minimal vertex covers of $U$. For instance, in the displayed example, the type of $R'/I(T)$ is 2, either by the decomposition $I(U)  = \langle v_1v_2,v_2v_3 \rangle = \langle v_1,v_3 \rangle \cap \langle v_2 \rangle$ or by the minimal vertex covers $\{v_1,v_3\}$ and $\{v_2\}$. The goal of this paper is to extend this computation to the following more general constructions.

    Paulsen and Sather-Wagstaff~\cite{MR3055580} generalized Villarreal's construction in on direction with the edge ideal of an edge-weighted graph $G_\omega$: the ideal $I(G_\omega)$ of $R$ which is ``generated by all weighted-edges of $G_\omega$'':
    \[I(G_\omega) = \left(X_i^{\omega(v_iv_j)}X_j^{\omega(v_iv_j)} \mathrel{\big |} v_iv_j\in E\right)R.\]
    In particular, if $\omega$ is the constant function defined by $\omega(v_iv_j) = 1$ for $v_iv_j \in E$, then $I(G_\omega) = I(G)$. 

    Building from Villarreal's work in another direction, Conca and De Negri~\cite{MR1666661} defined the \emph{$r$-path ideal} associated to $G$ to be the ideal $I_r(G)$ of $R$ that is ``generated by the paths in $G$ of length $r$'':
    \[I_r(G) = (X_{i_1} \cdots X_{i_{r+1}} \mid v_{i_1} \cdots v_{i_{r+1}} \text{ is a path of length $r$ in }G)R.\]
    In particular, if $r$ = 1, then $I_1(G) = I(G)$.

    Kubik and Sather-Wagstaff~\cite{MR3335988} gave a construction that subsumes each of these: the \emph{edge-weighted $r$-path ideal} $I_r(G_\omega)$ of an edge-weighted graph $G_\omega$; see Definition~\ref{defOfWRPathIdeal}\ref{defOfWRPathIdeala}. When $G$ is a tree, a theorem of Kubik and Sather-Wagstaff~\cite{MR3335988} gives a graph-theoretic characterization when $I_r(G_\omega)$ is Cohen-Macaulay. The main result of this paper computes the type for these ideals when they are Cohen-Macaulay. As with the above computation our result is purely graph-theoretical. These results are in Theorem~\ref{computeTypeOfWRPathOfWRPathSuspension} and Corollary~\ref{computeTypeOfWAndUMRPathOfRPathSuspensionCor}. They form the bulk of Section~\ref{sec2}. Necessary background information is collected in Section~\ref{sec1}.

\section{Background} \label{sec1} 

We begin this section with the definition of type. See, e.g., Bruns and Herzog~\cite{MR1251956} for undefined notions.

\begin{definition}
    Let $I$ be a proper homogeneous ideal of $R$ such that $I$ is Cohen-Macaulay. Let $d = \dim(R/I)$. The \emph{Cohen-Macaulay type}, denoted $r_R(R/I)$, is defined to be the dimension of the $\bbK$-vector space $\operatorname{Ext}_R^d(\bbK,R/I)$. In symbols: 
    \[r_R(R/I) = \dim_\bbK \operatorname{Ext}_R^d(\bbK,R/I).\]
\end{definition}

\begin{notation}
    For a monomial ideal of $I$ of $R$, let $\llbracket I \rrbracket$ the set of monomials contained in $I$.
\end{notation}

When visualizing an edge-weighted graph, we put a positive integer around an edge $v_iv_j$ to denote the weight of that edge, as follows.

\begin{example} \label{edge-weightedGraphExample}
    An edge-weighted graph $G_\omega$ with the vertex set $V = \{v_1,v_2,v_3,v_4,v_5\}$ and the edge set $E = \{v_1v_2,v_2v_3,v_2v_4,v_3v_4,v_4v_5\}$ is represented in the following drawing. 
    
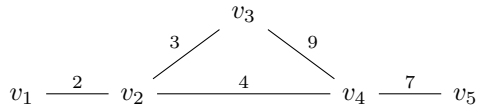
\[
        \begin{tikzcd}
            & & v_3 \ar[ld,dash,"3"'] \ar[rd,dash,"9"] \\
            v_1 \ar[r,dash,"2"] & v_2 \ar[rr,dash,"4"] && v_4 \ar[r,dash,"7"] & v_5
        \end{tikzcd}  
    \]
\end{example} 

The following definition provides a combinatorial description of decompositions of ideals constructed from edge-weighted graphs. See Section 2 of \cite{MR3335988}.

\begin{definition} \cite[Definitions 1.5 and 1.7]{MR3335988} \label{defOffWRPathVC}
    An \emph{edge-weighted $r$-path vertex cover} of $G_\omega$ is an ordered pair $(V',\delta')$ with $V' \subseteq V$ and $\delta':V \to \bbN$ such that $V'$ is an $r$-path vertex cover of $G$ and such that for any $r$-path $P_r := v_{i_1} \cdots v_{i_{r+1}}$ in $G$ at least one of the following holds:
    \begin{enumerate}
        \item $\delta'(v_{i_1}) \leq \omega(v_{i_1}v_{i_2})$;
        \item $\delta'(v_{i_{r+1}}) \leq \omega(v_{i_r}v_{i_{r+1}})$; or
        \item $\delta'(v_{i_j}) \leq \max(\omega(v_{i_{j-1}}v_{i_j}), \omega(v_{i_j} v_{i_{j+1}}))$ for some $j \in \{2,\dots,r\}$.
    \end{enumerate}
    The number $\delta'(v_i)$ is the \emph{weight} of $v_{i_j}$. \par 
    Given two edge-weighted $r$-path vertex covers $(V_1',\delta_1')$ and $(V_2',\delta_2')$ of $G_\omega$, we write $(V_2',\delta_2') \leq (V_1',\delta_1')$ if $V_2' \subseteq V_1'$ and $\delta_2'(v_i) \geq \delta_1'(v_i)$ for all $v_i \in V_2'$. An edge-weighted $r$-path vertex cover $(V',\delta')$ is \emph{minimal} if there does not exist another $f$-edge-weighted $r$-path vertex cover $(V'',\delta'')$ such that $(V'',\delta'') < (V',\delta')$.
\end{definition}

\begin{notation}
    For an edge-weighted $r$-path vertex cover $(V',\delta')$ of $G_\omega$, we also use the decorated set $\{v_i^{\delta'(v_i)} \mid v_i \in~V'\}$ to denote it, especially when we depict an edge-weighted $r$-path vertex cover of $G_\omega$ in sketches.
\end{notation}

We represent edge-weighted $r$-path vertex covers algebraically and diagrammatically, as follows.

\begin{example}
    Let $G_\omega$ be the edge-weighted graph from Example~\ref{edge-weightedGraphExample}. Then $\{v_2^3\}$ is an edge-weighted 3-path vertex cover of $G_\omega$, which is represented, as follows.
    
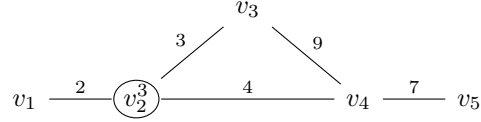
\[
        \begin{tikzcd}[every matrix/.append style={name=m},
            execute at end picture={
                \draw [<-] (m-2-2) ellipse (0.3cm and 0.25cm);
            }]
            & & v_3 \ar[ld,dash,"3"'] \ar[rd,dash,"9"] \\
            v_1 \ar[r,dash,"2"] & v_2^3 \ar[rr,dash,"4"] && v_4 \ar[r,dash,"7"] & v_5
        \end{tikzcd}  
    \]
    Since $\{v_2^4\}$ is not an edge-weighted 3-path vertex cover of $G_\omega$, $\{v_2^3\}$ is a minimal edge-weighted 3-path vertex cover. 
\end{example}

The main combinatorial objects we will work on are introduced below.
  
\begin{definition} \cite[Definition 3.4]{MR3335988} \label{defOfSuspensionWSp}
    The \emph{$r$-path suspension} of $G$ is the graph $\Sigma_rG$ obtained by adding a new path of length $r$ to each vertex of $G$ such that the vertex set
    \[V(\Sigma_r G) = \{v_{i,j} \mid i = 1,\ldots,d,j = 0,\ldots,r\} \text{ with } v_{i,0} = v_i, \fa i = 1,\ldots,d.\] 
    The new $r$-paths are called \emph{$r$-whiskers}. An \emph{edge-weighted $r$-path suspension} of $G_\omega$ is an edge-weighted graph $(\Sigma_rG)_\lambda$ with weight function $\lambda: \Sigma _rG \to \bbN$ such that the underlying graph $\Sigma_r G$ is an $r$-path suspension of $G$ and $\lambda(v_iv_j) = \omega(v_iv_j)$ for all $v_iv_j \in E(G)$, i.e., $\lambda|_{E(G)} = \omega$. 
\end{definition}
Examples of the 2-path suspension of an edge-weighted graph and an edge-weighted 2-path suspension of an edge-weighted graph are given by the following.
\begin{example} \label{examplesOfSuspensionsAndW}
    \begin{enumerate}
        \item \label{examplesOfSuspensionsAndWa} A 2-path suspension $\Sigma_2C_3$ of a $3$-cycle $C_3 = (\begin{tikzcd}v_1 \ar[r,dash] & v_2 \ar[r,dash] & v_3 \ar[r,dash] & v_1\end{tikzcd})$ is shown in the following. 
        
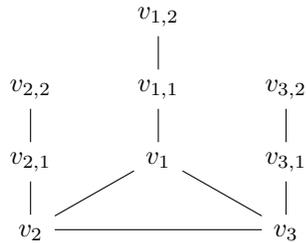
\[
            \begin{tikzcd}
                & v_{1,2} \ar[d,dash] \\ [-5pt]
                v_{2,2} \ar[d,dash] & v_{1,1} \ar[d,dash] & v_{3,2} \ar[d,dash] \\ [-5pt]
                v_{2,1} \ar[d,dash] & v_1 \ar[ld,dash] \ar[rd,dash] & v_{3,1} \ar[d,dash] \\ [-5pt]
                v_2 \ar[rr,dash] & & v_3
            \end{tikzcd} 
        \]
        \item \label{examplesOfSuspensionsAndWb} Let $(P_2)_\omega = (\begin{tikzcd}v_1 \ar[r,dash,"1"] & v_2 \ar[r,dash,"2"] & v_3\end{tikzcd})$ be an edge-weighted 2-path. An edge-weighted 2-path suspension $(\Sigma_2P_2)_\lambda$ is shown in the following. 
        \[
            \begin{tikzcd}
                v_1 \ar[d,dash,"1"'] \ar[r,dash,"4"] & v_{1,1} \ar[r,dash,"3"] & v_{1,2} \\
                v_2 \ar[d,dash,"2"'] \ar[r,dash,"3"] & v_{2,1} \ar[r,dash,"3"] & v_{2,2} \\
                v_3 \ar[r,dash,"2"] & v_{3,1} \ar[r,dash,"5"] & v_{3,2}
            \end{tikzcd}
        \]
    \end{enumerate}
\end{example}

Kubik and Sather-Wagstaff~\cite{MR3335988} introduced the edge-weighted $r$-path ideal of an edge-weighted graph.

\begin{definition} \label{defOfWRPathIdeal}
    \begin{enumerate}
        \item \label{defOfWRPathIdeala} The \emph{edge-weighted $r$-path ideal} associated to $G_\omega$ is the ideal $I_r(G_\omega)$ of $R$ that is ``generated by the maximal edge-weighted paths in $G$ of length $r$'':
            \[I_r(G_\omega) = \Bigggm(X_{i_1}^{e_{i_1}} \cdots X_{i_{r+1}}^{e_{i_{r+1}}} \mathrel{\Biggg |} \begin{array}{l}v_{i_1} \cdots v_{i_{r+1}} \text{ is a path in }G \text{ with }e_{i_1} = \omega(v_{i_1}v_{i_2}), \\ e_{i_j} = \max(\omega(v_{i_{j-1}}v_{i_j}),\omega(v_{i_j},v_{i_{j+1}})) \text{ for } 1 < j \leq r \\ \text{and } e_{i_{r+1}} = \omega(v_{i_r}v_{i_{r+1}})\end{array}\Bigggm)R.\]
        \item
            Let $(V',\delta')$ be a pair such that $V' \subseteq V$ and $\delta': V' \to \bbN$. We define the cardinality of $(V',\delta')$, denoted $\abs{(V',\delta')}$, to be the cardinality of $V'$. Set $P(V',\delta') \subseteq R$ to be the ideal ``generated by the elements of $(V',\delta')$'': 
            \[P(V',\delta') = \Bigl(X_i^{\delta'(v_i)} \mathrel{\big |} v_i \in V'\Bigr)R.\]
            Since $\bbK$ is a field, the ideals $P(V',\delta')$ are irreducible.
    \end{enumerate}
\end{definition}

\begin{remark}
    Let $(\Sigma_rG)_\lambda$ be an edge-weighted $r$-path suspension of $G_\omega$. Then the edge-weighted $r$-path ideal $I_r((\Sigma_rG)_\lambda)$ is an ideal of $R' = \bbK[\{X_{i,j} \mid i = 1,\ldots,d,j = 0,\ldots,r\}]$. 
\end{remark}

\begin{convention}
    Let $(\Sigma_rG)_\lambda$ be an edge-weighted $r$-path suspension of $G_\omega$. Then there is a bijection between vertices of $\Sigma_rG$ and variables of $R' = \bbK[X_{i,j} \mid i = 1,\dots,d,j = 0,\dots,r]$ given by $X_{i,j} \longleftrightarrow v_{i,j}$. 
    Based on the setting that $v_{i,0} = v_i$ for $i = 1,\dots,d$, we have that $X_{i,0} = X_i$ for $i = 1,\dots,d$. 
\end{convention}

In order to illustrate the previous concept and convention, it is helpful to consider one example.

\begin{example} \label{edge-weighted2PathIdealAndPIdealExample}
    Consider the following edge-weighted graph $(\Sigma_2 P_2)_\lambda$ as in Example~\ref{examplesOfSuspensionsAndW}\ref{examplesOfSuspensionsAndWb}. 
    \[    
        \begin{tikzcd}
            v_1 \ar[d,dash,"1"'] \ar[r,dash,"4"] & v_{1,1} \ar[r,dash,"3"] & v_{1,2} \\
            v_2 \ar[d,dash,"2"'] \ar[r,dash,"3"] & v_{2,1} \ar[r,dash,"3"] & v_{2,2} \\
            v_3 \ar[r,dash,"2"] & v_{3,1} \ar[r,dash,"5"] & v_{3,2}
        \end{tikzcd}
    \]
    The edge-weighted $2$-path ideal associated to $(\Sigma_2 P_2)_\lambda$ is 
    \begin{align*}
        I_2((\Sigma_2 P_2)_\lambda) &= (X_{1,2}^3X_{1,1}^4X_1^4,X_{1,1}^4X_1^4X_2,X_1X_2^3X_{2,1}^3,X_1X_2^2X_3^2,X_{2,2}^3X_{2,1}^3X_2^3,\\
        & \ \ \ \ \ \ X_{2,1}^3X_2^3X_3^2,X_2^2X_3^2X_{3,1}^2,X_{3,2}^5X_{3,1}^5X_3^2)R'.
    \end{align*}
    Let $V' = \{v_1,v_{2,1},v_3\} \subseteq V((\Sigma_2P_2)_\lambda)$ and $\delta': V' \longrightarrow \bbN$ be defined by $v_1 \longmapsto 3$, $v_{2,1} \longmapsto 2$ and $v_3 \longmapsto 4$. Then 
    \[P(V',\delta') = (X_1^3,X_{2,1}^2,X_3^4)R',\]
    where $R' = \bbK[X_1,X_{1,1},X_{1,2},X_2,X_{2,1},X_{2,2},X_3,X_{3,1},X_{3,2}]$.
\end{example}

The following concept will be use for turning the edge-weighted $r$-path ideal of an edge-weighted suspension $(\Sigma_rG)_\lambda$ into a monomial ideal in $R$. 

\begin{definition} \label{totallyProjectionMapP}
    Define a map $p:  R' \to R$ by sending $f$ to $g$, where $g$ is obtained by replacing every variable $X_{i,j}$ in $f$ with $X_i$. Let $I$ be a monomial ideal of $R'$. We set
    \[IR = p(I)R = (X_{i_1}^{a_1} \cdots X_{i_n}^{a_n} \in R \mid \ex X_{i_1,j_1}^{a_1} \cdots X_{i_n,j_n}^{a_n} \in \llbracket I \rrbracket)R.\]
    In words, $IR$ is the monomial ideal of $R$ obtained from $I$ by setting $X_{i,j} = X_i$ for all $i,j$. It is straightforward to show that if $f_1,\dots,f_m$ is a monomial generating sequence for $I$, then $p(f_1),\dots,p(f_m)$ is a monomial generating sequence for $IR$.
\end{definition}

To make this definition clearer, let's look at a concrete example.

\begin{example} \label{reductionOf2PathIdealOf2PathSuspensionOfP2}
    Consider the 2-path ideal $I_2(\Sigma_2P_2)$ from Example~\ref{edge-weighted2PathIdealAndPIdealExample}. Then
    \begin{align*}
        I_2((\Sigma_2 P_2)_\lambda)R &= (X_1^3X_1^4X_1^4,X_1^4X_1^4X_2,X_1X_2^3X_2^3,X_1X_2^2X_3^2,X_2^3X_2^3X_2^3,\\
        & \ \ \ \ \ \ X_2^3X_2^3X_3^2,X_2^2X_3^2X_3^2,X_3^5X_3^5X_3^2)R \\
        &= (X_1^{11},X_1^8X_2,X_1X_2^{6},X_1X_2^2X_3^2,X_2^9,X_2^6X_3^2,X_2^2X_3^4,X_3^{12})R. 
    \end{align*}
\end{example}

Consider the ideal $I_r((\Sigma_rG)_\lambda)$ in $R'$, where $(\Sigma_rG)_\lambda$ is an edge-weighted $r$-path suspension of an edge-weighted graph $G_\omega$. Then the ideal $I_r((\Sigma_rG)_\lambda)R$ in $R$ has some generators corresponding to the $r$-whiskers of $(\Sigma_rG)_\lambda$. From these generators, we can create a special ideal in $R$, given below.

\begin{definition}
    Let $(\Sigma_rG)_\lambda$ be an edge-weighted $r$-path suspension of $G_\omega$. Define
    \[\ffm^{[\underline a(\lambda)]} = (X_1^{a_1},\dots,X_d^{a_d})R,\]
    where for $i = 1,\dots,d$: $a_i = \sum_{k=0}^{r}e_{i,k}$ with
    \[e_{i,k} = \left\{\begin{array}{ll}
                \lambda(v_{i}v_{i,1}) & \text{if }k = 0, \\ 
                    \max(\lambda(v_{i,k-1} v_{i,k}), \lambda(v_{i,k}v_{i,k+1})) & \text{if }k = 1,\dots,r-1, \\
                    \lambda(v_{i,r-1}v_{i,r}) & \text{if }k = r.
        \end{array}\right.
    \]
\end{definition}

\begin{example}
    In Example~\ref{edge-weighted2PathIdealAndPIdealExample}, we have that $\ffm^{[\underline a(\lambda)]} = (X_1^{a_1},X_2^{a_2},X_3^{a_3})R$, where
    \begin{align*}
        a_1 &= \sum_{k=0}^2 e_{1,k} = \lambda(v_1v_{1,1}) + \max(\lambda(v_{1}v_{1,1}),\lambda(v_{1,1}v_{1,2})) + \lambda(v_{1,1}v_{1,2}) = 4 + 4 + 3 = 11, \\
        a_2 &= \sum_{k=0}^2 e_{2,k} = \lambda(v_2v_{2,1}) + \max(\lambda(v_{2}v_{2,1}),\lambda(v_{2,1}v_{2,2})) + \lambda(v_{2,1}v_{2,2}) = 3 + 3 + 3 = 9, \\
        a_3 &= \sum_{k=0}^2 e_{3,k} = \lambda(v_3v_{3,1}) + \max(\lambda(v_{3}v_{3,1}),\lambda(v_{3,1}v_{3,2})) + \lambda(v_{3,1}v_{3,2}) = 2 + 5 + 5 = 12.
    \end{align*}
\end{example}

\begin{example} \label{IrSrEIrSrM1PMExample}
    Based on Definitions~\ref{defOfSuspensionWSp},~\ref{defOfWRPathIdeal}, and~\ref{totallyProjectionMapP}, we observe that 
    \[I_r((\Sigma_rG)_\lambda)R = I_r((\Sigma_{r-1}G)_{\lambda'})R + \ffm^{[\underline a(\lambda)]}, \text{ where }\lambda' = \lambda|_{\Sigma_{r-1}G}.\]
\end{example}

The notation below together with Definition~\ref{defOffWRPathVC} provide a combinatorial description of decompositions of the ideal $I_r((\Sigma_{r-1}G)_{\lambda'})R$, where $\lambda' = \lambda|_{\Sigma_{r-1}G}$; see Theorem~\ref{wRPathIdealDecompOfRPATHSuspension}. 
\begin{notation}
    Let $(\Sigma_{r-1}G)_\lambda$ be an edge-weighted $(r-1)$-path suspension of $G_\omega$. We define a map $q: V((\Sigma_{r-1}G)_\lambda) \to V(G)$ as $q(v_{i,j}) = v_i$. Let $\ffP := (V'',\delta'')$ be such that $V'' \subseteq V((\Sigma_{r-1}G)_\lambda)$ and $\delta'': V'' \to \bbN$. Then
    \[q(V'') = \{v_i \mid \ex v_{i,j} \in V''\}.\]
    \par For each vertex $v_{i,j}$ in $(\Sigma_{r-1} G)_\lambda$, let $h_{i,j}$ be the maximal edge weight among all edges of $(\Sigma_{r-1}G)_\lambda$ adjacent to $v_{i,j}$:
    \[h_{i,j} = \max\{\lambda(v_{i,j}v) \mid v_{i,j}v \in E((\Sigma_{r-1}G)_\lambda)\}, \fa i = 1,\dots,d,\ j = 0,\dots,r-1\]
    \par For each $i$, we let $W_i(\ffP)$ be a set of vertices $v_{i,j}$ in $V''$ such that $\delta''(v_{i,j})$ is no larger than $h_{i,k}$:
    \[
        W_i(\ffP) = \{v_{i,j} \in V'' \mid \delta''(v_{i,j}) \leq h_{i,j}\}, \fa i = 1,\dots,d.
    \]
    \par We define a function $\gamma_{(V'',\delta'')}$ by 
    \begin{align*}
        \gamma_{(V'',\delta'')}: q(V'') &\longrightarrow \bbN \sqcup \{\infty\} \\
        v_i &\longmapsto \left\{\begin{array}{ll} 
                \min\Bigl\{\delta''(v_{i,j}) + \sum_{k = 0}^{j-1} h_{i,k} \mathrel{\Big |} v_{i,j} \in W_i(\ffP)\Bigr\} &\text{if } W_i(\ffP) \neq \emptyset, \\
                \infty &\text{otherwise.}
        \end{array}\right.
    \end{align*}
\end{notation}

Let's explore these four concepts further by considering a particular example.

\begin{example} \label{exampleOfWRPVSIffSubset}
    An edge-weighted 2-path suspension $(\Sigma_2P_1)_\lambda$ of $(P_1)_\omega = (\begin{tikzcd}v_1 \ar[r,dash,"2"] & v_2\end{tikzcd})$ with an edge-weighted 3-path vertex cover $\ffP := (V'',\delta'')$ of $(\Sigma_2P_1)_\lambda$ is given in the following sketch.
    
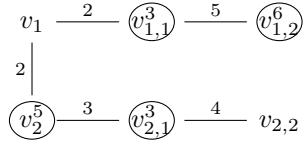
\[
        \begin{tikzcd}[every matrix/.append style={name=m},
          execute at end picture={
              \draw [<-] (m-1-2) ellipse (0.3cm and 0.25cm);
              \draw [<-] (m-1-3) ellipse (0.3cm and 0.25cm);
              \draw [<-] (m-2-1) ellipse (0.3cm and 0.25cm);
              \draw [<-] (m-2-2) ellipse (0.3cm and 0.25cm);
          }]
          v_1 \ar[d,dash,"2"'] \ar[r,dash,"2"] & v_{1,1}^3 \ar[r,dash,"5"] & v_{1,2}^6 \\
          v_2^5 \ar[r,dash,"3"] & v_{2,1}^3 \ar[r,dash,"4"] & v_{2,2}
        \end{tikzcd}
    \]
    Note that $q(V'') = \{v_1,v_2\}$. Since $\delta''(v_{1,1}) = 3 < 5 = \lambda(v_{1,1}v_{1,2})$ and $\delta''(v_{1,2}) = 6 > 5 = \lambda(v_{1,1}v_{1,2})$, we have that $W_1(\ffP) = \{v_{1,1}\}$. Similarly, $W_2(\ffP) = \{v_{2,1}\}$. Thus,
    \[\gamma_{(V'',\delta'')}(v_1) = \delta''(v_{1,1}) + \sum_{k=0}^{1-1}h_{1,k} = \delta''(v_{1,1}) + \max(\lambda(v_1v_2),\lambda(v_1v_{1,1})) = 3 + \max(2,2) = 5,\]
    \[\gamma_{(V'',\delta'')}(v_2) = \delta''(v_{2,1}) + \sum_{k=0}^{1-1}h_{2,k} = \delta''(v_{2,1}) + \max(\lambda(v_1v_2),\lambda(v_2v_{2,1})) = 3 + \max(2,3) = 6.\] 
\end{example}

\section{Edge-Weighted $r$-Path Ideals and the Type of $I_r((\Sigma_rG)_\lambda)$} \label{sec2}

This section is devoted to proving the main result of this paper, namely Theorem~\ref{computeTypeOfWRPathOfWRPathSuspension} and Corollary~\ref{computeTypeOfWAndUMRPathOfRPathSuspensionCor}.

\begin{proposition} \label{equivaDefForGammaVDeltaProp}
    Let $(\Sigma_{r-1}G)_\lambda$ be an edge-weighted $(r-1)$-path suspension of $G_\omega$. Let $\ffP := (V'',\delta'')$ be such that $V'' \subseteq V((\Sigma_{r-1}G)_\lambda)$ and $\delta'': V'' \to \bbN$. If $W_i(\ffP) \neq \emptyset$ for some $i \in \{1,\dots,d\}$, then for any $v_i \in q(V'')$,  
    \[\gamma_{(V'',\delta'')}(v_i) = \delta''(v_{i,j_0})+\sum_{k=0}^{j_0-1}h_{i,k}, \text{ where }j_0 := \min\{j \mid v_{i,j} \in W_i(\ffP)\}.\] 
\end{proposition}

\begin{proof}
    Suppose that $\delta''(v_{i,j}) + \sum_{k=0}^{j-1}h_{i,k} < \delta''(v_{i,j_0}) + \sum_{k=0}^{j_0-1}h_{i,k}$ for some $v_{i,j} \in W_i(\ffP)$. Then by the definition of $j_0$, we have that $j > j_0 \geq 0$ and so $\delta''(v_{i,j}) + \sum_{k=j_0}^{j-1} h_{i,k} < \delta''(v_{i,j_0})$.  Thus,
    \[h_{i,j_0} < \delta''(v_{i,j}) + h_{i,j_0} \leq \delta''(v_{i,j}) + \sum_{k=j_0}^{j-1} h_{i,k} < \delta''(v_{i,j_0}), \text{ i.e., }h_{i,j_0} < \delta''(v_{i,j_0}),\] 
    contradicting $v_{i,j_0} \in W_i(\ffP)$.
\end{proof}

The following theorem is a key for decomposing $I_r((\Sigma_{r-1}G)_{\lambda'})R$ with $\lambda' = \lambda|_{\Sigma_{r-1}G}$ and hence $I_r((\Sigma_rG)_\lambda)R$. The reader may wish to follow the argument with Example~\ref{exampleOfWRPVSIffSubset}.

\begin{theorem} \label{gammaWRPATHVCIff}
    Let $(\Sigma_{r-1}G)_\lambda$ be an edge-weighted $(r-1)$-path suspension of $G_\omega$ such that $\lambda(v_iv_j) \leq \lambda(v_i,v_{i,1})$ and $\lambda(v_iv_j) \leq \lambda(v_jv_{j,1})$ for all edges $v_iv_j \in E$. Let $\ffP := (V'',\delta'')$ be such that $V'' \subseteq V((\Sigma_{r-1}G)_\lambda)$ and $\delta'': V'' \to \bbN$. Then $I_r((\Sigma_{r-1}G)_\lambda)R \subseteq P(q(V''),\gamma_{(V'',\delta'')})$ if and only if $(V'',\delta'')$ is an edge-weighted $r$-path vertex cover of $(\Sigma_{r-1}G)_\lambda$. 
\end{theorem}

\begin{proof}
    $\Longrightarrow$ Assume that $I_r((\Sigma_{r-1}G)_\lambda)R \subseteq P(q(V''),\gamma_{(V'',\delta'')})$. Let $P_r := v_{p_1,q_1} \cdots v_{p_{r+1},q_{r+1}}$ be an $r$-path in $(\Sigma_{r-1}G)_\lambda$. Set 
    \[e_{p_k,q_k} = \left\{\begin{array}{ll}
                \lambda(v_{p_1,q_1}v_{p_2,q_2}) & \text{if }k = 1, \\ 
                    \max(\lambda(v_{p_{k-1},q_{k-1}} v_{p_k,q_k}), \lambda(v_{p_k,q_k}v_{p_{k+1},q_{k+1}})) & \text{if }k = 2,\dots,r, \\
                    \lambda(v_{p_{r},q_r}v_{p_{r+1},q_{r+1}}) & \text{if }k = r+1.
        \end{array}\right.
    \]
    Then $X_{p_1}^{e_{p_1,q_1}} \cdots X_{p_{r+1}}^{e_{p_{r+1},q_{r+1}}} \in \llbracket I_r ((\Sigma_{r-1}G)_\lambda)R \rrbracket \subseteq \llbracket P(q(V''),\gamma_{(V'',\delta'')}) \rrbracket$. Hence
    \[X_{i_0}^{\gamma_{(V'',\delta'')}(v_{i_0})} \mathrel{\big |}X_{p_1}^{e_{p_1,q_1}} \cdots X_{p_{r+1}}^{e_{p_{r+1},q_{r+1}}} \text{ for some }v_{i_0} \in q(V'').\]
    Then $v_{i_0} = v_{p_l}$ for some $l \in \{1,\ldots,r+1\}$, $\gamma_{(V'',\delta'')(v_{i_0})} < \infty$ and so
    \[\min\biggl\{\delta''(v_{i_0,j}) + \sum_{k=0}^{j-1}h_{i_0,k} \mathrel{\Big |} v_{i_0,j} \in W_{i_0}(\ffP)\biggr\} = \gamma_{(V'',\delta'')}(v_{i_0}) \leq \sum_{k=0}^{r+1} \idca_k \cdot e_{p_k,q_k},\] 
    where $\idca_k$ is an indicator function given by
    \[\idca_k = \left\{\begin{array}{ll}1 & \text{if }p_k = i_0, \\ 0 & \text{otherwise.}\end{array}\right.\]
    Hence $v_{p_l} = v_{i_0} \in q(V'')$. Since $P_r$ is an $r$-path in the $(r-1)$-path suspension $\Sigma_{r-1}G$, it is either an $r$-path in $G$, or it is a path with one end located in $G$ and the other located in one of the $(r-1)$-whiskers, or it is a path with two ends located two of the $(r-1)$ whiskers, respectively. Therefore, $P_r$ in $\Sigma_{r-1}G$ is of the following form diagrammatically.
    \[
        \begin{tikzcd}
            v_{p_1,0} \ar[d,dash] \ar[r,dash] &[-5pt] v_{p_1,1} \ar[r,dash] &[-5pt] \cdots \ar[r,dash] &[-5pt] v_{p_1,q_1} \\
            \vdots \ar[d,dash] \\
            v_{p_{r+1},0} \ar[r,dash] & v_{p_{r+1},1} \ar[r,dash] & \cdots \ar[r,dash] & v_{p_{r+1},q_{r+1}}
        \end{tikzcd}
    \]
    where $q_1$ or $q_{r+1}$ could be 0. Let $M_0 := \max_{1 \leq k \leq r+1}\{q_k \mid i_0 = p_k\}$. Then     
    \[
        M_0 = \left\{\begin{array}{ll}
                q_1 & \text{ if }i_0 = p_1, \\
                q_{r+1} & \text{ if }i_0 = p_{r+1}. 
        \end{array} \right.
    \]
    Since $\gamma_{(V'',\delta'')}(v_i) < \infty$, we have that $W_{i_0}(\ffP) \neq \emptyset$. Set $j_0 := \min\{j \mid v_{i_0,j} \in W_{i_0}(\ffP)\}$. Then by Proposition~\ref{equivaDefForGammaVDeltaProp},
    \begin{equation} \label{wRPRedWR}
        \delta''(v_{i_0,j_0}) + \sum_{k=0}^{j_0-1}h_{i_0,k} = \min\biggl\{\delta''(v_{i_0,j}) + \sum_{k=0}^{j-1}h_{i_0,k}\mathrel{\Big |}v_{i_0,j} \in V''\biggr\} \leq \sum_{k=0}^{r+1} \idca_k \cdot e_{p_k,q_k} = \sum_{k=0}^{M_0}e_{i_0,k}.
    \end{equation}
    Suppose that $j_0 > M_0$. Then since $e_{i_0,k} \leq h_{i_0,k}$ for $k = 0,\dots,M_0$, by~\eqref{wRPRedWR}, we have that
    \[\delta''(v_{i_0,j_0}) + \sum_{k=0}^{M_0}e_{i_0,k} \leq \delta''(v_{i_0,j_0}) + \sum_{k=0}^{M_0}h_{i_0,k} \leq \delta''(v_{i_0,j_0}) + \sum_{k=0}^{j_0-1}h_{i_0,k} \leq \sum_{k=0}^{M_0}e_{i_0,k}, \text{ i.e., }\delta''(v_{i_0,j_0}) \leq 0,\]
    contradicting $\delta''(v_{i_0,j_0}) \geq 1$. Hence $j_0 \leq M_0$. Also, we have that there must exist a sub-path of $P_r$ of the form 
    \[
        \begin{tikzcd}
            v_{i_0,0} \ar[r,dash] & v_{i_0,1} \ar[r,dash] & \cdots \ar[r,dash] & v_{i_0,M_0},
        \end{tikzcd}
    \]
    so there exists a vertex in this path of the form $v_{i_0,j_0} = v_{p_k,q_k}$ for some $k$ in $\{1,\ldots,r+1\}$. Hence $v_{p_k,q_k} = v_{i_0,j_0} \in W_i(\ffP) \subseteq V''$ and $v_{i_0,j_0} \in V(P_r)$.
    \begin{enumerate}
        \item Assume that $0 = j_0 < M_0$. Since $\lambda(v_iv_j) \leq \min\{\lambda(v_i,v_{i,1}),\lambda(v_j,v_{j,1})\}$ for all edges $v_iv_j \in E$ and $M_0 \geq 1$, we have that $e_{i_0,0} = \lambda(v_{i_0,0}v_{i_0,1}) = h_{i_0,0}$. Since $v_{i_0,j_0} \in W_{i_0}(\ffP)$, we have that $\delta''(v_{i_0,0}) \leq h_{i_0,0} = e_{i_0,0}$. 
        \item Assume that $0 < j_0 < M_0$. Since $v_{i_0,j_0} \in W_{i_0}(\ffP)$, we have that
            \[\delta''(v_{i_0,j_0}) \leq \max(\lambda(v_{i_0,j_0-1}v_{i_0,j_0}),\lambda(v_{i_0,j_0}v_{i_0,j_0+1})) = e_{i_0,j_0}.\]
        \item 
            Assume that $j_0 = M_0$. Since $e_{i_0,k} \leq h_{i_0,k}$ for $k = 0,\dots,j_0-1$, by~\eqref{wRPRedWR}, we have that
            \[\delta''(v_{i_0,j_0}) + \sum_{k=0}^{j_0-1}e_{i_0,k} \leq \delta''(v_{i_0,j_0}) + \sum_{k=0}^{j_0-1}h_{i_0,k} \leq \sum_{k=0}^{M_0}e_{i_0,k} = \sum_{k=0}^{j_0}e_{i_0,k}, \text{ i.e., }\delta''(v_{i_0,j_0}) \leq e_{i_0,j_0}.\]
    \end{enumerate}
    Thus, we have that $v_{i_0,j_0} \in V'' \cap V(P_r)$ and $\delta''(v_{i_0,j_0}) \leq e_{i_0,j_0}$ by (a), (b), and (c). Thus, $(V'',\delta'')$ is an edge-weighted $r$-path vertex cover of $(\Sigma_{r-1}G)_\lambda$. \par
    $\Longleftarrow$ Assume that $(V'',\delta'')$ is an edge-weighted $r$-path vertex cover of $(\Sigma_{r-1}G)_\lambda$. We need to show that every monomial generator of $I_r((\Sigma_{r-1}G)_\lambda)R$ is in $P(q(V''),\gamma_{(V'',\delta'')})$. Let $\underline X^{\underline b} := X_{i_1}^{e_{i_1,j_1}} \cdots X_{i_{r+1}}^{e_{i_{r+1},j_{r+1}}}$ be a generator corresponding to the $r$-path $P_r := v_{i_1,j_1} \cdots v_{i_{r+1},j_{r+1}}$ in $(\Sigma_{r-1}G)_\lambda$. We need to show that $ \underline X^{\underline b} \in P(q(V''),\gamma_{(V'',\delta'')})$. Note that $P_r$ in $\Sigma_{r-1}G$ is of the following form.
    \[
        \begin{tikzcd}
            v_{i_1,0} \ar[d,dash] \ar[r,dash] &[-5pt] v_{i_1,1} \ar[r,dash] &[-5pt] \cdots \ar[r,dash] &[-5pt] v_{i_1,j_1} \\
            \vdots \ar[d,dash] \\
            v_{i_{r+1},0} \ar[r,dash] & v_{i_{r+1},1} \ar[r,dash] & \cdots \ar[r,dash] & v_{i_{r+1},j_{r+1}}
        \end{tikzcd}
    \]
    where $j_1$ or $j_{r+1}$ may be 0. Since $P_r$ is an $r$-path in $(\Sigma_{r-1}G)_\lambda$ and $(V'',\delta'')$ is an edge-weighted $r$-path vertex cover of $(\Sigma_{r-1}G)_\lambda$, we have that there exists some $l \in \{1,\ldots,r+1\}$ such that $v_{i_l,j_l} \in V''$ and $\delta''(v_{i_l,j_l}) \leq e_{i_l,j_l}$. Hence $v_{i_l,j_l} \in W_{i_l}(\ffP)$ and then
    \[\gamma_{(V'',\delta'')}(v_{i_l}) = \min \biggl\{\delta''(v_{i_l,t}) + \sum_{k = 0}^{t-1} h_{i_l,k}\mathrel{\Big |} v_{i_l,t} \in W_{i_l}(\ffP)\biggr\} \leq \delta''(v_{i_l,j_l}) + \sum_{k = 0}^{j_l-1} h_{i_l,k}.\]
    Let $M_0 := \max_{1 \leq k \leq r+1}\{j_k \mid i_l = i_k\}$. Then $j_l \leq M_0$. Therefore,
    \[\delta''(v_{i_l,j_l}) + \sum_{k = 0}^{j_l-1} e_{i_l,k} \leq e_{i_l,j_l} + \sum_{k=0}^{j_l-1}e_{i_l,k} = \sum_{k=0}^{j_l}e_{i_l,k} \leq \sum_{k=0}^{M_0}e_{i_l.k}, \text{ i.e., } \delta''(v_{i_l,j_l}) + \sum_{k=0}^{j_l-1}e_{i_l,k} \leq \sum_{k=0}^{M_0}e_{i_l,k}.\] 
    \begin{enumerate}
        \item Assume that $j_l = 0$. Then
            \[\gamma_{(V'',\delta'')}(v_{i_l}) \leq \delta''(v_{i_l,j_l}) + \sum_{k = 0}^{j_l-1} h_{i_l,k} = \delta''(v_{i_l,j_l}) + \sum_{k = 0}^{j_l-1} e_{i_l,k} \leq \sum_{k=0}^{M_0}e_{i_l,k} = \sum_{k=0}^{r+1} \idca_{l,k} \cdot e_{i_k,j_k} = b_{i_l},\]
            where $\idca_{l,k} = \left\{\begin{array}{ll}1 & \text{if }j_k = j_l, \\ 0 & \text{otherwise,}\end{array}\right. \fa k = 1,\ldots,r+1$. 
        \item Assume that $j_l > 0$. Then $M_0 \geq 1$. Since $\lambda(v_iv_j) \leq \min\{\lambda(v_i,v_{i,1}),\lambda(v_j,v_{j,1})\}$ for all edges $v_iv_j \in E$, we have that $e_{i_0,0} = \lambda(v_{i_0,0}v_{i_0,1}) = h_{i_0,0}$. Also, since $e_{i_0,k} = h_{i_0,k}$ for $k = 1,\dots,j_l-1$, we have that 
            \[\gamma_{(V'',\delta'')}(v_{i_l}) \leq \delta''(v_{i_l,j_l}) + \sum_{k = 0}^{j_l-1} h_{i_l,k} = \delta''(v_{i_l,j_l}) + \sum_{k = 0}^{j_l-1} e_{i_l,k} \leq \sum_{k=0}^{M_0}e_{i_l,k} = \sum_{k=0}^{r+1} \idca_{l,k} \cdot e_{i_k,j_k} = b_{i_l}.\]
    \end{enumerate}
    Thus, $X_{i_l}^{\gamma_{(V'',\delta'')}(v_{i_l})} \mathrel{\big |}\underline X^{\underline b}$ by (a) and (b). Thus, $\underline X^{\underline b} \in P(q(V''),\gamma_{(V'',\delta'')})$.
\end{proof}

\begin{proposition} \label{writeIntersectionOfWRPIdealWithPqVGammaProp}
    Let $(\Sigma_{r-1}G)_\lambda$ be an edge-weighted $(r-1)$-path suspension of $G_\omega$ such that $\lambda(v_iv_j) \leq \lambda(v_i,v_{i,1})$ and $\lambda(v_iv_j) \leq \lambda(v_j,v_{j,1})$ for all $v_iv_j \in E$. The monomial ideal $I_r((\Sigma_{r-1}G)_{\lambda})R$ can be written as a finite intersection of irreducible ideals of the form $P(q(V'') := \{v_{i_1},\ldots,v_{i_t}\},\gamma_{(V'',\delta'')})$ with $V'' \subseteq V(\Sigma_{r-1}G)$ and $\delta'': V'' \to \bbN$. 
\end{proposition}

\begin{proof}
    \cite[Theorem 7.5.1]{MR3839602} gives a decomposition of $I_r((\Sigma_{r-1}G)_{\lambda})R$. Let $(X_{b_1}^{\beta_{b_1}},\dots,X_{b_s}^{\beta_{b_s}})R$ occur in that decomposition. Without loss of generality, assume $b_1,\dots,b_s \in \bbN$ are distinct. We claim that for each exponent $\beta_{b_k}$ we can find at least one $r$-path $P_r$ in $(\Sigma_{r-1}G)_\lambda$ such that the weights on $P_r$ are related to it. Let $k \in \{1,\dots,s\}$. Then by \cite[Theorem 7.5.1]{MR3839602}, there exists a generator $p(X_{i_1,j_1}^{e_{i_1,j_1}} \cdots X_{i_{r+1},j_{r+1}}^{e_{i_{r+1},j_{r+1}}})$ of $I_r((\Sigma_{r-1}G)_{\lambda})R$ with $v_{i_1,j_1} \cdots v_{i_{r+1},j_{r+1}}$ an $r$-path in $(\Sigma_{r-1}G)_{\lambda}$ such that 
    \[\beta_{b_k} = 
        \left\{\begin{array}{ll}
                e_{b_k,0} &\text{if } M_k = 0, \\
                \lambda(v_{b_k,M_k}v_{b_k,M_k-1}) + \sum_{l = 0}^{M_k-1} h_{b_k,l} & \text{if }M_k \geq 1, \\
        \end{array}\right.
    \]
    where $M_k:= \max_{1 \leq n \leq r+1} \{j_n \mid b_k = i_n\} \leq r-1$ and
    \[e_{i_m,j_m} = \left\{\begin{array}{ll}
                \lambda(v_{i_1,j_1}v_{i_{2},j_{2}}) & \text{if }m = 1, \\ 
                    \max(\lambda(v_{i_{m-1},j_{m-1}} v_{i_m,j_m}), \lambda(v_{i_m,j_m}v_{i_{m+1},j_{m+1}})) & \text{if }m = 2,\dots,r, \\
                    \lambda(v_{i_{r},j_r}v_{i_{r+1},j_{r+1}}) & \text{if }m = r+1.
        \end{array}\right.
    \]
    Repeat the process for each $k \in \{1,\dots,s\}$ and set $V'' = \{v_{b_1,M_1},\dots,v_{b_s,M_s}\}$. Then $q(V'') = \{v_{b_1},\dots,v_{b_s}\}$. Define 
    \begin{align*}
        \delta'': V'' &\longrightarrow \bbN\\
        v_{b_k,M_k} &\longmapsto 
        \left\{\begin{array}{ll}
                \lambda(v_{b_k,M_k}v_{b_k,M_k-1}) & \text{if } M_k \geq 1, \\
                \beta_{b_k} (= e_{b_k,0}) &\text{if } M_k = 0,
        \end{array}\right. \fa k = 1,\dots,s. 
    \end{align*}
    We claim that $(X_{b_1}^{\beta_{b_1}},\dots,X_{b_s}^{\beta_{b_s}})R = P(q(V''),\gamma_{(V'',\delta'')})$. It suffices to show that $\gamma_{(V'',\delta'')}(v_{b_k}) = \beta_{b_k}$ for $k = 1,\dots,s$. Let $\ffP := (V'',\delta'')$. By definition, we have that $W_k(\ffP) = \{v_{b_k,M_k}\}$ for $k = 1,\dots,s$.     
    \begin{enumerate}
        \item Assume that $M_k \geq 1$. Then
            \[\gamma_{(V'',\delta'')}(v_{b_k}) = \delta''(v_{b_k,M_k}) + \sum_{l=0}^{M_k-1} h_{b_k,l} = \lambda(v_{b_k,M_k}v_{b_k,M_k-1}) + \sum_{l=0}^{M_k-1} h_{b_k,l} = \beta_{b_k}.\]
        \item Assume that $M_k = 0$. Then
            \[\gamma_{(V'',\delta'')}(v_{b_k}) = \delta''(v_{b_k,0}) + \sum_{l=0}^{0-1} h_{b_k,l} = \delta''(v_{b_k,0}) = \beta_{b_k}. \qedhere\]
    \end{enumerate}
\end{proof}

We use the following example to illustrate the previous theorem and its proof.

\begin{example} 
    Consider the following graph $(\Sigma_2P_1)_\lambda$ as in Example~\ref{exampleOfWRPVSIffSubset}.
    \[
        \begin{tikzcd}
          v_1 \ar[d,dash,"2"'] \ar[r,dash,"2"] & v_{1,1} \ar[r,dash,"5"] & v_{1,2} \\
          v_2 \ar[r,dash,"3"] & v_{2,1} \ar[r,dash,"4"] & v_{2,2}
        \end{tikzcd}
    \]
    Since $I_3((\Sigma_2P_1)_\lambda) = (X_{1,2}^5X_{1,1}^5X_1^{2}X_2^2,X_{1,1}^2X_1^2X_2^3X_{2,1}^3,X_1^2X_2^3X_{2,1}^4X_{2,2}^4)R'$, we have that 
    \[I_3((\Sigma_2P_1))R = (X_1^{12}X_2^2,X_1^4X_2^6,X_1^2X_2^{11})R.\] 
    By \cite[Theorem 7.5.1]{MR3839602}, we have an irreducible but not necessarily an irreducible decomposition
    \begin{align*}
        I_3((\Sigma_2P_1))R &= (X_1^{12},X_1^4,X_1^2)R \cap (X_1^{12},X_1^4,X_2^{11})R \cap (X_1^{12},X_2^6,X_1^2)R \cap (X_1^{12},X_2^6,X_2^{11})R \\ 
        & \ \ \ \ \cap (X_2^{2},X_1^4,X_1^2)R \cap (X_2^{2},X_1^4,X_2^{11})R \cap (X_2^{2},X_2^6,X_1^2)R \cap (X_2^{2},X_2^6,X_2^{11})R \\ 
        &= (X_1^2)R \cap (X_1^4,X_2^{11})R \cap (X_1^2,X_2^6)R \cap (X_1^{12},X_2^6)R \\
        &\ \ \ \ \cap (X_1^2,X_2^2)R \cap (X_1^4,X_2^2)R \cap (X_1^2,X_2^2)R \cap (X_2^2)R.
    \end{align*}
    Consider the monomial ideal $(X_1^4,X_{2}^{11})R$. Then $b_1 = 1,b_2 = 2$, $\beta_{b_1} = \beta_1 = 4$ and $\beta_{b_2} = \beta_2 = 11$. Consider the generator $X_{1,1}^{e_{1,1}}X_{1,0}^{e_{1,0}}X_{2,0}^{e_{2,0}}X_{2,1}^{e_{2,1}} := X_{11}^2X_1^2X_2^3X_{2,1}^3$ of $I_3((\Sigma_2P_1)_\lambda)$. Then $M_1 := 1$ and $\beta_1 = \lambda(v_{1,1}v_{1,0}) + h_{1,0} = 2 + 2 = 4$. Consider the generator $X_{1,0}^{e_{1,0}}X_{2,0}^{e_{2,0}}X_{2,1}^{e_{2,1}}X_{2,2}^{e_{2,2}} := X_1^2X_2^3X_{2,1}^4X_{2,2}^4$ of $I_3((\Sigma_2P_1)_\lambda)$. Then $M_2 := 2$ and $\beta_2 = \lambda(v_{2,2}v_{2,1}) + h_{2,0} + h_{2,1} = 4 + 3 + 4 = 11$. Let $V'' = \{v_{1,1},v_{2,2}\}$ and $\delta'':V'' \longrightarrow \bbN$ given by $v_1 \longmapsto \lambda(v_{1,1}v_{1,0}) = 2$ and $v_{2,2} \longmapsto \lambda(v_{2,2}v_{2,1}) = 4$. Then $q(V'') = \{v_1,v_2\}$, $\gamma_{(V'',\delta'')}(v_1) = \delta''(v_{1,1}) + h_{1,0} = 2+2$ and $\gamma_{(V'',\delta'')}(v_2) = \delta''(v_{2,2}) + h_{2,0} + h_{2,1} = 4 + 3 + 4 = 11$. Hence
    \[P(q(V''),\gamma{(V'',\delta'')}) = P(\{v_1^4,v_2^{11}\}) = (X_1^4,X_2^{11})R.\]
\end{example}

    The next result gives our first decomposition needed for computing $\operatorname{type}(R'/I_r((\Sigma_rG)_\lambda))$.

\begin{theorem} \label{wRPathIdealDecompOfRPATHSuspension}
    Let $(\Sigma_rG)_\lambda$ be an edge-weighted $r$-path suspension of $G_\omega$ such that $\lambda(v_iv_j) \leq \lambda(v_iv_{i,1})$ and $\lambda(v_iv_j) \leq \lambda(v_jv_{j,1})$ for all edges $v_iv_j \in E$. One has 
    \[I_r((\Sigma_{r-1} G)_{\lambda'})R = \bigcap_{(V'',\delta'') \text{ w. $r$-path v. cover of }(\Sigma_{r-1}G)_{\lambda'}}P(q(V''),\gamma_{(V'',\delta'')}), \text{ where } \lambda' = \lambda|_{\Sigma_{r-1}G},\]
    and
    \[I_r((\Sigma_r G)_\lambda)R = \bigcap_{(V'',\delta'') \text{ w. $r$-path v. cover of }(\Sigma_{r-1}G)_{\lambda'}}P(q(V''),\gamma_{(V'',\delta'')})+\ffm^{[\underline a(\lambda)]}.\]
\end{theorem}

\begin{proof}
    Since $I_r((\Sigma_rG)_\lambda)R = I_r((\Sigma_{r-1}G)_{\lambda'})R + \ffm^{[\underline a(\lambda)]}$ by Example~\ref{IrSrEIrSrM1PMExample}, we have that by \cite[Theorem 7.5.3]{MR3839602}, it is enough to show that
    \[I_r((\Sigma_{r-1}G)_{\lambda'})R = \bigcap_{(V'',\delta'') \text{ w. $r$-path v. cover of }(\Sigma_{r-1}G)_{\lambda'}} P(q(V''),\gamma_{(V'',\delta'')}).\]
    By Proposition~\ref{writeIntersectionOfWRPIdealWithPqVGammaProp}, the monomial ideal $I_r((\Sigma_{r-1}G)_{\lambda'})R$ can be written as a finite intersection of irreducible ideals of the form $P(q(V'') := \{v_{i_1},\ldots,v_{i_t}\},\gamma_{(V'',\delta'')})$ with $V'' \subseteq V(\Sigma_{r-1}G)$ and $\delta'': V'' \to \bbN$. Then by Theorem~\ref{gammaWRPATHVCIff},
    \begin{align*}
        I_r((\Sigma_{r-1}G)_{\lambda'})R &\subseteq \bigcap_{(V'',\delta'') \text{ w. $r$-path v. cover of }(\Sigma_{r-1}G)_{\lambda'}} P(q(V''),\gamma_{(V'',\delta'')}) \\
        &\subseteq \bigcap_{(V'',\delta'')\text{ w. $r$-path v. cover of }(\Sigma_{r-1}G)_{\lambda'} \text{ in decomp. of } I_r((\Sigma_{r-1}G)_{\lambda'})R} P(q(V''),\gamma_{(V'',\delta'')}) \\
        &= I_r((\Sigma_{r-1}G)_{\lambda'})R.
    \end{align*}
    Therefore,
    \[I_r((\Sigma_{r-1}G)_{\lambda'})R = \bigcap_{(V'',\delta'') \text{ w. $r$-path v. cover of }(\Sigma_{r-1}G)_{\lambda'}} P(q(V''),\gamma_{(V'',\delta'')}). \qedhere \]
\end{proof}

The following example illustrates the previous theorem.

\begin{example}
    Consider the following edge-weighted 3-path suspension $(\Sigma_3P_1)_\lambda$ of the edge-weighted 1-path $(P_1)_\omega = (\begin{tikzcd}v_1 \ar[r,dash,"2"] & v_2\end{tikzcd})$. 
    \[
        \begin{tikzcd}
            v_1 \ar[d,dash,"2"'] \ar[r,dash,"2"] & v_{1,1} \ar[r,dash,"5"] & v_{1,2} \ar[r,dash,"2"] & v_{1,3} \\
            v_2 \ar[r,dash,"3"] & v_{2,1} \ar[r,dash,"4"] & v_{2,2} \ar[r,dash,"2"] & v_{2,3}
        \end{tikzcd}
    \]
    Let $\lambda' = \lambda|_{\Sigma_{2}P_1}$. Since $I_3((\Sigma_2P_1)_{\lambda'}) = (X_{1,2}^5X_{1,1}^5X_1^2X_2^2,X_{1,1}^2X_1^2X_2^3X_{2,1}^3,X_1^2X_2^3X_{2,1}^4X_{2,2}^4)$, by Theorem~\ref{wRPathIdealDecompOfRPATHSuspension}, we have two infinite intersections: 
    \[I_3((\Sigma_2P_1)_{\lambda'})R = (X_1^{12}X_2^2,X_1^4X_2^6,X_1^2X_2^{11})R = \bigcap_{(V'',\delta'') \text{ w. $r$-path v. cover of }(\Sigma_{2}P_1)_{\lambda'}}P(q(V''),\gamma_{(V'',\delta'')}),\]
    and
    \begin{align*}
        I_3((\Sigma_3P_1)_{\lambda'})R &= (X_1^{12}X_2^2,X_1^4X_2^6,X_1^2X_2^{11})R + (X_1^{14},X_2^{13})R \\
        &= \bigcap_{(V'',\delta'') \text{ w. $r$-path v. cover of }(\Sigma_{2}P_1)_{\lambda'}}P(q(V''),\gamma_{(V'',\delta'')}).
    \end{align*}
\end{example}

    The next result is key for our second decomposition result, Corollary~\ref{minDecomOfReductiveWRPathIdealRMinus1}.

\begin{lemma} \label{usedForWRPathVCRedundantLemma}
    Let $\ffp := (V_1'',\delta_1''), \ffP := (V_2'',\delta_2'')$ be such that $V_1'',V_2'' \subseteq V((\Sigma_{r-1}G)_\lambda)$ and $\delta_1'',\delta_2'': V'' \to \bbN$. If $(V_1'',\delta_1'') \leq (V_2'',\delta_2'')$, then $P(q(V_1''),\gamma_{(V_1'',\delta_1'')}) \subseteq P(q(V_2''),\gamma_{(V_2'',\delta_2'')})$.
\end{lemma}

\begin{proof}
    Let $X_i^{\gamma_{(V_1'',\delta_1'')}(v_i)}$ be a nonzero generator of $P(q(V_1''),\gamma_{(V_1'',\delta_1'')})$. Then $\gamma_{(V_1'',\delta_1'')}(v_i) < \infty$, and so $W_{i}(\ffp) \neq \emptyset$. Given $(V_1'',\delta_1'') \leq (V_2'',\delta_2'')$, we have that $W_i(\ffp) \subseteq W_i(\ffP)$, so $W_i(\ffP) \neq \emptyset$. Note that $v_i \in q(V_1'') \subseteq q(V_2'')$ since $V_1'' \subseteq V_2''$, and so $X_i^{\gamma_{(V_2'',\delta_2'')}(v_i)} \in P(q(V_2''),\gamma_{(V_2'',\delta_2'')})$. Since $\delta_1'' \geq \delta_2''|_{V_1''}$, we have that
    \begin{align*}
        \gamma_{(V_1'',\delta_1'')}(v_i) &= \min\biggl\{\delta_1''(v_{i,t})+\sum_{k=0}^{t-1} h_{i,k} \mathrel{\Big |} v_{i,t} \in W_{i_j}(\ffp)\biggr\} \\
        &\geq \min\biggl\{\delta_2''(v_{i,t})+\sum_{k=0}^{t-1} h_{i,k} \mathrel{\Big |} v_{i,t} \in W_{i}(\ffP)\biggr\} = \gamma_{(V_2'',\delta_2'')}(v_i).
    \end{align*}
    Hence $X_i^{\gamma_{(V_2'',\delta_2'')}(v_i)} \mathrel{\big |} X_i^{\gamma_{(V_1'',\delta_1'')}(v_i)}$. Thus, $P(q(V_1''),\gamma_{(V_1'',\delta_1'')}) \subseteq P(q(V_2''),\gamma_{(V_2'',\delta_2'')})$.
\end{proof}

The following example illustrates the previous lemma.

\begin{example}  
    Consider the following two pairs of sets $\ffp := (V_1'',\delta_1'') := \{v_{1,1}^4,v_2^5,v_{2,1}^6\}$ and $\ffP := (V_2'',\delta_2'') := \{v_{1,1}^3,v_{1,2}^6,v_2^5,v_{2,1}^3\}$ of $(\Sigma_2P_1)_\lambda$ as in Example~\ref{exampleOfWRPVSIffSubset}. 
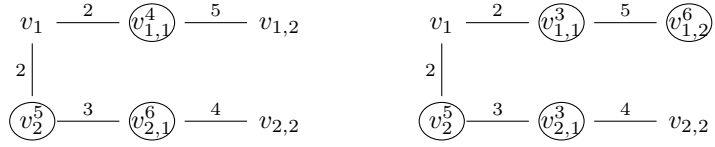
    \[
        \begin{tikzcd}[every matrix/.append style={name=m},
          execute at end picture={
              \draw [<-] (m-1-2) ellipse (0.3cm and 0.25cm);
              \draw [<-] (m-2-1) ellipse (0.3cm and 0.25cm);
              \draw [<-] (m-2-2) ellipse (0.3cm and 0.25cm);
          }]
          v_1 \ar[d,dash,"2"'] \ar[r,dash,"2"] & v_{1,1}^4 \ar[r,dash,"5"] & v_{1,2} \\
          v_2^5 \ar[r,dash,"3"] & v_{2,1}^6 \ar[r,dash,"4"] & v_{2,2}
        \end{tikzcd} \ \ \ \ \ \ \ \ \ \ \ \ 
        \begin{tikzcd}[every matrix/.append style={name=m},
          execute at end picture={
              \draw [<-] (m-1-2) ellipse (0.3cm and 0.25cm);
              \draw [<-] (m-1-3) ellipse (0.3cm and 0.25cm);
              \draw [<-] (m-2-1) ellipse (0.3cm and 0.25cm);
              \draw [<-] (m-2-2) ellipse (0.3cm and 0.25cm);
          }]
          v_1 \ar[d,dash,"2"'] \ar[r,dash,"2"] & v_{1,1}^3 \ar[r,dash,"5"] & v_{1,2}^6 \\
          v_2^5 \ar[r,dash,"3"] & v_{2,1}^3 \ar[r,dash,"4"] & v_{2,2}
        \end{tikzcd}
    \]
    Since $V_1'' \subseteq V_2''$ and $\delta_1'' \geq \delta_2''|_{V_1''}$, we have that $(V_1'',\delta_1'') \leq (V_2'',\delta_2'')$. Similar to Example~\ref{exampleOfWRPVSIffSubset}, we have that $W_1(\ffp) = \{v_{1,1}\}$ and $W_2(\ffp) = \emptyset$. Hence $\gamma_{(V_1'',\delta_1'')}(v_2) = \infty$ and
    \[\gamma_{(V_1'',\delta_1'')}(v_1) = \delta_1''(v_{1,1}) + \sum_{k=0}^{1-1}h_{1,k} = \delta_1''(v_{1,1}) + \max(\lambda(v_1v_2),\lambda(v_1v_{1,1}))  = 4 + \max(2,2) = 5.\]
    Also, since $q(V_1'') = \{v_1,v_2\}$, we have that $P(q(V_1''),\gamma_{(V_1'',\delta_1'')}) = (X_1^5,X_2^\infty)R = (X_1^5)R$. Then from Example~\ref{exampleOfWRPVSIffSubset} we have that
    \[P(q(V_2''),\gamma_{(V_2'',\delta_2'')}) = (X_1^5,X_2^6)R \supseteq (X_1^5)R = P(q(V_1''),\gamma_{(V_1'',\delta_1'')}).\]
\end{example}

    Here is our second decomposition result for computing $\operatorname{type}(R'/I_r((\Sigma_rG)_\lambda))$.

\begin{corollary} \label{minDecomOfReductiveWRPathIdealRMinus1}
    Let $(\Sigma_rG)_\lambda$ be an edge-weighted $r$-path suspension of $G_\omega$ such that $\lambda(v_iv_j) \leq \lambda(v_i,v_{i,1})$ and $\lambda(v_iv_j) \leq \lambda(v_j,v_{j,1})$ for all edges $v_iv_j \in E$. One has 
    \[I_r((\Sigma_{r-1}G)_{\lambda'})R = \bigcap_{(V'',\delta'') \text{ min. w. $r$-path v. cover of }(\Sigma_{r-1}G)_{\lambda'}} P(q(V''),\gamma_{(V'',\delta'')}), \text{ where }\lambda' = \lambda|_{\Sigma_{r-1}G},\]
    and
    \[I_r((\Sigma_rG)_\lambda)R = \bigcap_{(V'',\delta'') \text{ min. w. $r$-path v. cover of }(\Sigma_{r-1}G)_{\lambda'}} P(q(V''),\gamma_{(V'',\delta'')}) + \ffm^{[\underline a(\lambda)]}.\]
\end{corollary}

\begin{proof}
    By Example~\ref{IrSrEIrSrM1PMExample} and \cite[Theorem 7.5.3]{MR3839602}, it is enough to prove that
    \[I_r((\Sigma_{r-1}G)_{\lambda'})R = \bigcap_{(V'',\delta'') \text{ min. w. $r$-path v. cover of }(\Sigma_{r-1}G)_{\lambda'}} P(q(V''),\gamma_{(V'',\delta'')}).\]
    By Theorem~\ref{wRPathIdealDecompOfRPATHSuspension}, it is enough to show that
    \begin{align*}
        &\ \ \ \ \bigcap_{(V'',\delta'') \text{ edge-weighted. $r$-path v. cover of }(\Sigma_{r-1}G)_{\lambda'}} P(q(V''),\gamma_{(V'',\delta'')}) \\
        &= \bigcap_{(V'',\delta'') \text{ min. edge-weighted $r$-path v. cover of }(\Sigma_{r-1}G)_{\lambda'}} P(q(V''),\gamma_{(V'',\delta'')}).
    \end{align*}
    \par ``$\subseteq$'' follows because every minimal edge-weighted $r$-path vertex cover is an edge-weighted $r$-path vertex cover. \par
    ``$\supseteq$'' follows from \cite[Lemma 1.11]{MR3335988} and Lemma~\ref{usedForWRPathVCRedundantLemma}.
\end{proof}

\begin{example} \label{exampleOfDecomOfWRPRIdeal}
    Consider the following edge-weighted 3-path suspension $(\Sigma_3P_1)_\lambda$ of the edge-weighted 1-path $(P_1)_\omega = (\begin{tikzcd}v_1 \ar[r,dash,"2"] & v_2\end{tikzcd})$. 
    \[
        \begin{tikzcd}
            v_1 \ar[d,dash,"2"'] \ar[r,dash,"2"] & v_{1,1} \ar[r,dash,"5"] & v_{1,2} \ar[r,dash,"2"] & v_{1,3} \\
            v_2 \ar[r,dash,"3"] & v_{2,1} \ar[r,dash,"4"] & v_{2,2} \ar[r,dash,"2"] & v_{2,3}
        \end{tikzcd}
    \]
    We depict the minimal edge-weighted 3-path vertex covers of $(\Sigma_2P_1)_{\lambda'}$ with $\lambda' = \lambda|_{\Sigma_2P_1}$ in the following sketches.
    \[
        \begin{tikzcd}[every matrix/.append style={name=m},
          execute at end picture={
              \draw [<-] (m-1-1) ellipse (0.3cm and 0.25cm);
          }]
            v_1^2 \ar[d,dash,"2"'] \ar[r,dash,"2"] & v_{1,1} \ar[r,dash,"5"] & v_{1,2} \\
            v_2 \ar[r,dash,"3"] & v_{2,1} \ar[r,dash,"4"] & v_{2,2} 
        \end{tikzcd} \ \ \ \ \ \ \ \
        \begin{tikzcd}[every matrix/.append style={name=m},
          execute at end picture={
              \draw [<-] (m-2-1) ellipse (0.3cm and 0.25cm);
          }]
            v_1 \ar[d,dash,"2"'] \ar[r,dash,"2"] & v_{1,1} \ar[r,dash,"5"] & v_{1,2} \\
            v_2^2 \ar[r,dash,"3"] & v_{2,1} \ar[r,dash,"4"] & v_{2,2} 
        \end{tikzcd} \ \ \ \ \ \ \ \
        \begin{tikzcd}[every matrix/.append style={name=m},
          execute at end picture={
              \draw [<-] (m-1-2) ellipse (0.3cm and 0.25cm);
              \draw [<-] (m-2-1) ellipse (0.3cm and 0.25cm);
          }]
            v_1 \ar[d,dash,"2"'] \ar[r,dash,"2"] & v_{1,1}^5 \ar[r,dash,"5"] & v_{1,2} \\
            v_2^3 \ar[r,dash,"3"] & v_{2,1} \ar[r,dash,"4"] & v_{2,2} 
        \end{tikzcd}
    \]
    \[
        \begin{tikzcd}[every matrix/.append style={name=m},
          execute at end picture={
              \draw [<-] (m-1-3) ellipse (0.3cm and 0.25cm);
              \draw [<-] (m-2-1) ellipse (0.3cm and 0.25cm);
          }]
            v_1 \ar[d,dash,"2"'] \ar[r,dash,"2"] & v_{1,1} \ar[r,dash,"5"] & v_{1,2}^5 \\
            v_2^3 \ar[r,dash,"3"] & v_{2,1} \ar[r,dash,"4"] & v_{2,2} 
        \end{tikzcd} \ \ \ \ \ \ \ \ 
        \begin{tikzcd}[every matrix/.append style={name=m},
          execute at end picture={
              \draw [<-] (m-1-2) ellipse (0.3cm and 0.25cm);
              \draw [<-] (m-2-2) ellipse (0.3cm and 0.25cm);
          }]
            v_1 \ar[d,dash,"2"'] \ar[r,dash,"2"] & v_{1,1}^2 \ar[r,dash,"5"] & v_{1,2} \\
            v_2 \ar[r,dash,"3"] & v_{2,1}^4 \ar[r,dash,"4"] & v_{2,2} 
        \end{tikzcd} \ \ \ \ \ \ \ \
        \begin{tikzcd}[every matrix/.append style={name=m},
          execute at end picture={
              \draw [<-] (m-1-2) ellipse (0.3cm and 0.25cm);
              \draw [<-] (m-2-2) ellipse (0.3cm and 0.25cm);
          }]
            v_1 \ar[d,dash,"2"'] \ar[r,dash,"2"] & v_{1,1}^5 \ar[r,dash,"5"] & v_{1,2} \\
            v_2 \ar[r,dash,"3"] & v_{2,1}^3 \ar[r,dash,"4"] & v_{2,2} 
        \end{tikzcd}
    \]
    \[
        \begin{tikzcd}[every matrix/.append style={name=m},
          execute at end picture={
              \draw [<-] (m-1-3) ellipse (0.3cm and 0.25cm);
              \draw [<-] (m-2-2) ellipse (0.3cm and 0.25cm);
          }]
            v_1 \ar[d,dash,"2"'] \ar[r,dash,"2"] & v_{1,1} \ar[r,dash,"5"] & v_{1,2}^5 \\
            v_2 \ar[r,dash,"3"] & v_{2,1}^3 \ar[r,dash,"4"] & v_{2,2} 
        \end{tikzcd} \ \ \ \ \ \ \ \
        \begin{tikzcd}[every matrix/.append style={name=m},
          execute at end picture={
              \draw [<-] (m-1-2) ellipse (0.3cm and 0.25cm);
              \draw [<-] (m-2-3) ellipse (0.3cm and 0.25cm);
          }]
            v_1 \ar[d,dash,"2"'] \ar[r,dash,"2"] & v_{1,1}^2 \ar[r,dash,"5"] & v_{1,2} \\
            v_2^2 \ar[r,dash,"3"] & v_{2,1} \ar[r,dash,"4"] & v_{2,2}^4 
        \end{tikzcd}
    \]
    Since $I_3((\Sigma_2P_1)_{\lambda'}) = (X_{1,2}^5X_{1,1}^5X_1^2X_2^2,X_{1,1}^2X_1^2X_2^3X_{2,1}^3,X_1^2X_2^3X_{2,1}^4X_{2,2}^4)$, by Corollary~\ref{minDecomOfReductiveWRPathIdealRMinus1}, we have that
    \begin{align*}
        I_3((\Sigma_2P_1)_{\lambda'})R &= (X_1^{12}X_2^2,X_1^4X_2^6,X_1^2X_2^{11})R = (X_1^2)R \cap (X_2^2)R \cap (X_1^7,X_2^3)R \cap (X_1^{12},X_2^3)R \\
    & \ \ \ \ \cap (X_1^4,X_2^7) \cap (X_1^7,X_2^6) \cap (X_1^{12},X_2^6)R \cap (X_1^4,X_2^{11})R.
    \end{align*}
    Thus, the first decomposition in Corollary~\ref{minDecomOfReductiveWRPathIdealRMinus1} may be redundant.
\end{example}

    In light of the preceding example, we define another order from which we can produce an irredundant decomposition. Lemma~\ref{MMinWRPathVCLemma} is the key for understanding how this ordering helps with irredundancy.

\begin{definition} 
    Given two minimal edge-weighted $r$-path vertex covers $(V_1'',\delta_1''),(V_2'',\delta_2'')$ of $(\Sigma_{r-1}G)_\lambda$, we write $(V_1'',\delta_1'') \leq_{\mathcal p} (V_2'',\delta_2'')$ if $q(V_1'') \subseteq q(V_2'')$ and $\gamma_{(V_1'',\delta_1'')} \geq \gamma_{(V_2'',\delta_2'')}|_{q(V_1'')}$. A minimal edge-weighted $r$-path vertex cover $(V'',\delta'')$ is \emph{$\mathcal p$-minimal} if there is not another minimal edge-weighted $r$-path vertex cover $(V''',\delta''')$ such that $(V'',\delta'') <_{\mathcal p} (V''',\delta''')$.
\end{definition}

\begin{example} \label{pMinimalExample}
    Consider the following two minimal edge-weighted 3-path vertex covers $(V_1'',\delta_1'') := \{v_{1,1}^5,v_2^3\}$ and $(V_2'',\delta_2'') := \{v_{1,2}^5,v_2^3\}$ of $(\Sigma_2P_1)_\lambda$ as in Example~\ref{exampleOfWRPVSIffSubset}. 
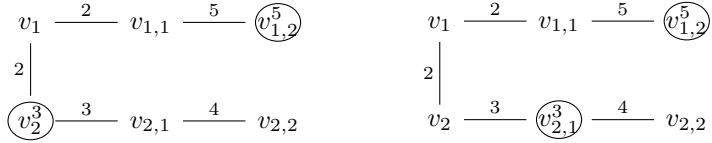
    \[
        \begin{tikzcd}[every matrix/.append style={name=m},
          execute at end picture={
              \draw [<-] (m-1-3) ellipse (0.3cm and 0.25cm);
              \draw [<-] (m-2-1) ellipse (0.3cm and 0.25cm);
          }]
            v_1 \ar[d,dash,"2"'] \ar[r,dash,"2"] & v_{1,1} \ar[r,dash,"5"] & v_{1,2}^5 \\
            v_2^3 \ar[r,dash,"3"] & v_{2,1} \ar[r,dash,"4"] & v_{2,2} 
        \end{tikzcd} \ \ \ \ \ \ \ \ \ \ \ \ 
        \begin{tikzcd}[every matrix/.append style={name=m},
          execute at end picture={
              \draw [<-] (m-1-3) ellipse (0.3cm and 0.25cm);
              \draw [<-] (m-2-2) ellipse (0.3cm and 0.25cm);
          }]
            v_1 \ar[d,dash,"2"'] \ar[r,dash,"2"] & v_{1,1} \ar[r,dash,"5"] & v_{1,2}^5 \\
            v_2 \ar[r,dash,"3"] & v_{2,1}^3 \ar[r,dash,"4"] & v_{2,2} 
        \end{tikzcd}
    \]
    Then $q(V_1'') = \{v_1,v_2\} = q(V_2'')$. Since 
    \[\gamma_{(V_1'',\delta_1'')}(v_1) = \delta_1''(v_{1,2}) + h_{1,1} + h_{1,0} = 5 + 5 + 2 = \delta_2''(v_{1,2}) + h_{1,1} + h_{1,0} = \gamma_{(V_2'',\delta_2'')}(v_1),\]
    and $\gamma_{(V_1'',\delta_1'')}(v_2) = \delta_1''(v_2) = 3 < 3 + 3 = \delta_2''(v_{2,1}) + h_{2,0} = \gamma_{(V_2'',\delta_2'')}(v_2)$, we have that $\gamma_{(V_1'',\delta_1'')} < \gamma_{(V_2'',\delta_2'')}$. Thus, $(V_1'',\delta_1'') >_{\mathcal p} (V_2'',\delta_2'')$. Hence $(V_1'',\delta_1'')$ is not $\mathcal p$-minimal. 
\end{example}

\begin{lemma} \label{carEqualityOfMinAndPMinLemmaW}
    Let $\ffp := (W',\delta'), \ffP := (W'',\delta'')$ be two minimal edge-weighted $r$-path vertex covers of $(\Sigma_{r-1}G)_\lambda$ such that $(W'',\delta'') \leq_{\mathcal p} (W',\delta')$, then $\abs{(W'',\delta'')}  = \abs{(W',\delta')}$ and $q(W'') = q(W')$. 
\end{lemma}

\begin{proof}
    Since $W'$ is a minimal $r$-path vertex cover of $\Sigma_{r-1}G$, for distinct $v_{i_1,j_1},v_{i_2,j_2} \in W'$, we have that $i_1 \neq i_2$. Also, since $q(W'') \subseteq q(W')$, $\abs{W''} = \abs{q(W'')} \leq \abs{q(W')} = \abs{W'}$. Suppose that $\abs{W''} < \abs{W'}$. Then there exists $v_{i,j} \in W'$ such that $v_i \not \in q(W'')$. Since $W'$ is a minimal $r$-path vertex cover of $\Sigma_{r-1}G$, there is an $r$-path $P_r$ in $\Sigma_{r-1}G$ that can only be covered by $v_{i,j}$. By assumption, $P_r$ can be covered by some $v_{k,l} \in W''$, so $v_k \in q(W'')$. Also, since $v_i \not \in q(W'')$, we have that $k \neq i$. Since $\gamma_{W'}(v_k) \leq \gamma_{W''}(v_k)$, we have that $v_{k,t} \in W'$ for some $t = \gamma_{W'}(v_k) \leq \gamma_{W''}(v_k) = l$. Note that $P_r$ can also be covered by $v_{k,t} \in W'$, a contradiction. Hence $\abs{W''} = \abs{W'}$ and thus $\abs{q(W'')} = \abs{q(W')}$. Since $q(W'') \subseteq q(W')$, we have that $q(W'') = q(W')$.
\end{proof}

The following example illustrates the previous lemma.

\begin{example}
    Consider the following two minimal edge-weighted 3-path vertex covers $(V_1'',\delta_1'') := \{v_{1,1}^5,v_2^3\}$ and $(V_2'',\delta_2'') := \{v_{1,2}^5,v_2^3\}$ of $(\Sigma_2P_1)_\lambda$ as in Example~\ref{pMinWRPVCByCriterionAndComputePMinExample}\ref{pMinWRPVCByCriterionAndComputePMinExamplea}. 
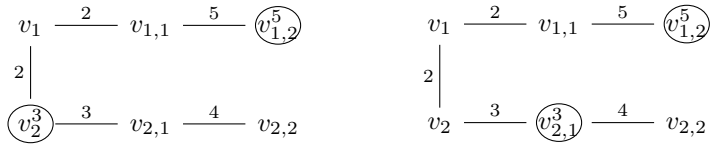
    \[
        \begin{tikzcd}[every matrix/.append style={name=m},
          execute at end picture={
              \draw [<-] (m-1-3) ellipse (0.3cm and 0.25cm);
              \draw [<-] (m-2-1) ellipse (0.3cm and 0.25cm);
          }]
            v_1 \ar[d,dash,"2"'] \ar[r,dash,"2"] & v_{1,1} \ar[r,dash,"5"] & v_{1,2}^5 \\
            v_2^3 \ar[r,dash,"3"] & v_{2,1} \ar[r,dash,"4"] & v_{2,2} 
        \end{tikzcd} \ \ \ \ \ \ \ \ \ \ \ \ 
        \begin{tikzcd}[every matrix/.append style={name=m},
          execute at end picture={
              \draw [<-] (m-1-3) ellipse (0.3cm and 0.25cm);
              \draw [<-] (m-2-2) ellipse (0.3cm and 0.25cm);
          }]
            v_1 \ar[d,dash,"2"'] \ar[r,dash,"2"] & v_{1,1} \ar[r,dash,"5"] & v_{1,2}^5 \\
            v_2 \ar[r,dash,"3"] & v_{2,1}^3 \ar[r,dash,"4"] & v_{2,2} 
        \end{tikzcd}
    \]
    By Example~\ref{pMinWRPVCByCriterionAndComputePMinExample}\ref{pMinWRPVCByCriterionAndComputePMinExamplea}, $(V_1'',\delta_1'') < _{\mathcal p} (V_2'',\delta_2'')$. Then $\abs{(V_1'',\delta_1'')} = \abs{\{v_{1,2},v_2\}} = 2 = \abs{\{v_{1,2},v_{2,1}\}} =\abs{(V_2'',\delta_2'')}$ and $q(V_1'') = \{v_1,v_2\} = q(V_2'')$.
\end{example}

    The following theorem can be used as an algorithm to find the set of $\mathcal p$-minimal edge-weighted $r$-path vertex covers of $(\Sigma_{r-1}G)_\lambda$ from the set of minimal edge-weighted $r$-path vertex covers.

\begin{theorem} \label{criterionForPMinimalityThm}
    Let $\ffp := (V_1'',\delta_1''), \ffP := (V_2'',\delta_2'')$ be two minimal edge-weighted $r$-path vertex covers of $(\Sigma_{r-1}G)_\lambda$. Then $(V_1'',\delta_1'') \leq_{\mathcal p} (V_2'',\delta_2'')$ if and only if $q(V_1'') = q(V_2'')$ and for any $v_{i_l} \in q(V_1'')$: $j_{1,l} > j_{2,l}$ or $j_{1,l} = j_{2,l}$ and $\delta_1''(v_{i_l,j_{1,l}}) \geq \delta_2''(v_{i_l,j_{2,l}})$ with $j_{1,l} := \{j \mid v_{i_l,j} \in V_1''\}$ and $j_{2,l} = \{j \mid v_{i_l,j} \in V_2''\}$. 
\end{theorem}

\begin{proof}
    By Lemma~\ref{carEqualityOfMinAndPMinLemmaW}, $(V_1'',\delta_1'') \leq_{\mathcal p} (V_2'',\delta_2'')$ if and only if $q(V_1'') = q(V_2'')$ and $\gamma_{(V_1'',\delta_1'')}|_{q(V_1'')} \geq \gamma_{(V_2'',\delta_2'')}|_{q(V_1'')}$ if and only if $q(V_1'') = q(V_2'')$ and for any $v_{i_l} \in q(V_1'')$, $\gamma_{(V_1'',\delta_1'')}(v_{i_l}) \geq \gamma_{(V_2'',\delta_2'')}(v_{i_l})$ if and only if $q(V_1'') = q(V_2'')$ and for any $v_{i_l} \in q(V_1'')$, $\delta_1''(v_{i_l,j_{1,l}})+\sum_{k=0}^{j_{1,l}-1}h_{i_l,k} \geq \delta_2''(v_{i_l,j_{2,l}})+\sum_{k=0}^{j_{2,l}-1}h_{i_l,k}$ by Proposition~\ref{equivaDefForGammaVDeltaProp}. We claim that For $v_{i_l} \in q(V_1'')$, $\delta_1''(v_{i_l,j_{1,l}})+\sum_{k=0}^{j_{1,l}-1}h_{i_l,k} \geq \delta_2''(v_{i_l,j_{2,l}})+\sum_{k=0}^{j_{2,l}-1}h_{i_l,k}$ if and only if $j_{1,l} > j_{2,l}$, or $j_{1,l} = j_{2,l}$ and $\delta_1''(v_{i_l,j_{1,l}}) \geq \delta_2''(v_{i_l,j_{2,l}})$. Then we are done. \par 
    $\Longleftarrow$ Assume that $j_{1,l} > j_{2,l}$, or $j_{1,l} = j_{2,l}$ and $\delta_1''(v_{i_l,j_{1,l}}) \geq \delta_2''(v_{i_l,j_{2,l}})$. Then
        \[\alpha := \bigg(\delta_1''(v_{i_l,j_{1,l}}) + \sum_{k=0}^{j_{1,l}-1}h_{i_l,k}\bigg) - \biggl(\delta_2''(v_{i_l,j_{2,l}}) + \sum_{k=0}^{j_{2,l}-1}h_{i_l,k}\biggr) = \delta_1''(v_{i_l,j_{1,l}}) - \delta_2''(v_{i_l,j_{2,l}}) + \sum_{k = j_{2,l}}^{j_{1,l}-1} h_{i_l,k}.\]
        To prove our statement, it is equivalent to show that $\alpha \geq 0$.
        \begin{enumerate}
            \item If $j_{1,l} > j_{2,l}$, then $\alpha \geq \delta_1''(v_{i_l,j_{1,l}}) - \delta_2''(v_{i_l,j_{2,l}}) + h_{i_l,j_{2,l}} > h_{i_l,j_{2,l}} - \delta_2''(v_{i_l,j_{2,l}}) \geq 0$.
            \item If $j_{1,l} = j_{2,l}$ and $\delta_1''(v_{i_l,j_{1,l}}) \geq \delta_2''(v_{i_l,j_{2,l}})$, then $\alpha = \delta_1''(v_{i_l,j_{1,l}}) - \delta_2''(v_{i_l,j_{2,l}}) \geq 0$.
        \end{enumerate}
        \par $\Longrightarrow$ Suppose that $j_{1,l} < j_{2,l}$, or $j_{1,l} = j_{2,l}$ and $\delta_1''(v_{i_l,j_{1,l}}) < \delta_2''(v_{i_l,j_{2,l}})$. Then
        \[\alpha := \bigg(\delta_1''(v_{i_l,j_{1,l}}) + \sum_{k=0}^{j_{1,l}-1}h_{i_l,k}\bigg) - \biggl(\delta_2''(v_{i_l,j_{2,l}}) + \sum_{k=0}^{j_{2,l}-1}h_{i_l,k}\biggr) = \delta_1''(v_{i_l,j_{1,l}}) - \delta_2''(v_{i_l,j_{2,l}}) - \sum_{k = j_{1,l}}^{j_{2,l}-1} h_{i_l,k}.\]
        To prove our statement, it is equivalent to show that $\alpha < 0$.
        \begin{enumerate}
            \item If $j_{1,l} = j_{2,l}$ and $\delta_1''(v_{i_l,j_{1,l}}) < \delta_2''(v_{i_l,j_{2,l}})$, then $\alpha = \delta_1''(v_{i_l,j_{1,l}}) - \delta_2''(v_{i_l,j_{2,l}}) < 0$. 
            \item Assume that $j_{1,l} < j_{2,l}$. Since $v_{i_l,j_{1,l}} \in V_1''$ and $V_1''$ is a minimal edge-weighted $r$-path vertex cover, $\delta_1''(v_{i_l},j_{1,l}) \leq h_{i_l,j_{1,l}}$. Hence 
            \[\alpha = \delta_1''(v_{i_l,j_{1,l}}) - \delta_2''(v_{i_l,j_{2,l}}) - \sum_{k = j_{1,l}}^{j_{2,l}-1} h_{i_l,k} < \delta_1''(v_{i_l,j_{1,l}}) - h_{i_l,j_{1,l}} \leq 0. \qedhere\]
        \end{enumerate}
\end{proof}

The following example illustrates the previous theorem.

\begin{example} \label{pMinWRPVCByCriterionAndComputePMinExample}
    We have the following examples.
    \begin{enumerate}
        \item \label{pMinWRPVCByCriterionAndComputePMinExamplea}
            Consider the following two minimal edge-weighted 3-path vertex covers $(V_1'',\delta_1'') := \{v_{1,1}^5,v_2^3\}$ and $(V_2'',\delta_2'') := \{v_{1,2}^5,v_2^3\}$ of $(\Sigma_2P_1)_\lambda$ as in Example~\ref{pMinimalExample}. 
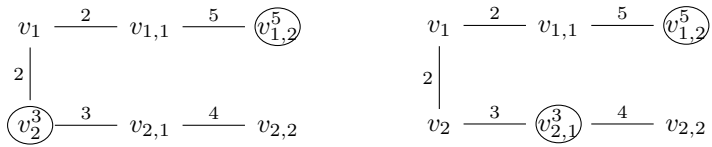
            \[
                \begin{tikzcd}[every matrix/.append style={name=m},
                  execute at end picture={
                      \draw [<-] (m-1-3) ellipse (0.3cm and 0.25cm);
                      \draw [<-] (m-2-1) ellipse (0.3cm and 0.25cm);
                  }]
                    v_1 \ar[d,dash,"2"'] \ar[r,dash,"2"] & v_{1,1} \ar[r,dash,"5"] & v_{1,2}^5 \\
                    v_2^3 \ar[r,dash,"3"] & v_{2,1} \ar[r,dash,"4"] & v_{2,2} 
                \end{tikzcd} \ \ \ \ \ \ \ \ \ \ \ \ 
                \begin{tikzcd}[every matrix/.append style={name=m},
                  execute at end picture={
                      \draw [<-] (m-1-3) ellipse (0.3cm and 0.25cm);
                      \draw [<-] (m-2-2) ellipse (0.3cm and 0.25cm);
                  }]
                    v_1 \ar[d,dash,"2"'] \ar[r,dash,"2"] & v_{1,1} \ar[r,dash,"5"] & v_{1,2}^5 \\
                    v_2 \ar[r,dash,"3"] & v_{2,1}^3 \ar[r,dash,"4"] & v_{2,2} 
                \end{tikzcd}
            \]
            Then $q(V_1'') = \{v_{i_1} := v_1,v_{i_2} := v_2\} = q(V_2'')$. Note that
            \[j_{1,1} = \min\{j \mid v_{i_1,j} \in V_1''\} = \min\{j \mid v_{1,j} \in V_1''\} = 2,\]
            \[j_{1,2} = \min\{j \mid v_{i_2,j} \in V_1''\} = \min\{j \mid v_{2,j} \in V_1''\} = 0,\]
            \[j_{2,1} = \min\{j \mid v_{i_1,j} \in V_2''\} = \min\{j \mid v_{2,j} \in V_2''\} = 2,\]
            \[j_{2,2} = \min\{j \mid v_{i_2,j} \in V_2''\} = \min\{j \mid v_{2,j} \in V_2''\} = 1.\]
            Since $j_{1,1} = 2 = j_{2,1}$ and $\delta_1''(v_{1,j_{1,1}}) = \delta_1''(v_{1,2}) = 5 = \delta_2''(v_{1,2}) = \delta_2''(v_{1,j_{2,1}})$, and $j_{1,2} = 0 < 1 = j_{2,2}$, we have that $(V_1'',\delta_1'') <_{\mathcal p} (V_2'',\delta_2'')$ by Theorem~\ref{criterionForPMinimalityThm}. 
        \item \label{pMinWRPVCByCriterionAndComputePMinExampleb}
            Consider all the minimal edge-weighted 3-path vertex covers of $(\Sigma_2P_1)_{\lambda'}$ as in Example~\ref{exampleOfDecomOfWRPRIdeal}. Apply Theorem~\ref{criterionForPMinimalityThm} repeatedly, we get all the $\mathcal p$-minimal edge-weighted 3-path vertex covers in the following.
            
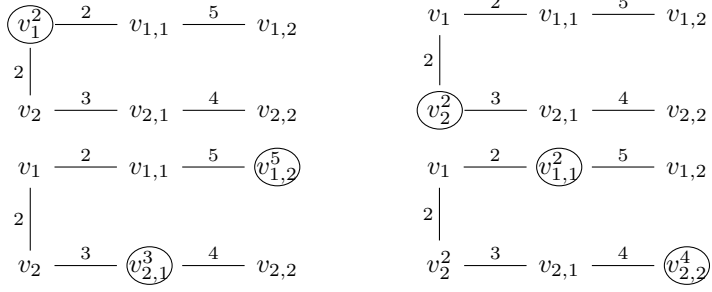
\[
                \begin{tikzcd}[every matrix/.append style={name=m},
                  execute at end picture={
                      \draw [<-] (m-1-1) ellipse (0.3cm and 0.25cm);
                  }]
                    v_1^2 \ar[d,dash,"2"'] \ar[r,dash,"2"] & v_{1,1} \ar[r,dash,"5"] & v_{1,2} \\
                    v_2 \ar[r,dash,"3"] & v_{2,1} \ar[r,dash,"4"] & v_{2,2}
                \end{tikzcd} \ \ \ \ \ \ \ \ \ \ \ \ 
                \begin{tikzcd}[every matrix/.append style={name=m},
                  execute at end picture={
                      \draw [<-] (m-2-1) ellipse (0.3cm and 0.25cm);
                  }]
                    v_1 \ar[d,dash,"2"'] \ar[r,dash,"2"] & v_{1,1} \ar[r,dash,"5"] & v_{1,2} \\
                    v_2^2 \ar[r,dash,"3"] & v_{2,1} \ar[r,dash,"4"] & v_{2,2}
                \end{tikzcd}
            \]
            \[
                \begin{tikzcd}[every matrix/.append style={name=m},
                  execute at end picture={
                      \draw [<-] (m-1-3) ellipse (0.3cm and 0.25cm);
                      \draw [<-] (m-2-2) ellipse (0.3cm and 0.25cm);
                  }]
                    v_1 \ar[d,dash,"2"'] \ar[r,dash,"2"] & v_{1,1} \ar[r,dash,"5"] & v_{1,2}^5 \\
                    v_2 \ar[r,dash,"3"] & v_{2,1}^3 \ar[r,dash,"4"] & v_{2,2}
                \end{tikzcd} \ \ \ \ \ \ \ \ \ \ \ \ 
                \begin{tikzcd}[every matrix/.append style={name=m},
                  execute at end picture={
                      \draw [<-] (m-1-2) ellipse (0.3cm and 0.25cm);
                      \draw [<-] (m-2-3) ellipse (0.3cm and 0.25cm);
                  }]
                    v_1 \ar[d,dash,"2"'] \ar[r,dash,"2"] & v_{1,1}^2 \ar[r,dash,"5"] & v_{1,2} \\
                    v_2^2 \ar[r,dash,"3"] & v_{2,1} \ar[r,dash,"4"] & v_{2,2}^4
                \end{tikzcd}
            \]
    \end{enumerate}
\end{example}

The next two results are key for our third and final decomposition result.

\begin{proposition} \label{existenceOfMMinWRPathVCOfAnyMProp}
    For every minimal edge-weighted $r$-path vertex cover $\ffp := (W',\delta')$ of $(\Sigma_{r-1}G)_\lambda$, there is a $\mathcal p$-minimal edge-weighted $r$-path vertex cover $(W'',\delta'')$ of $(\Sigma_{r-1}G)_\lambda$ such that $(W'',\delta'') \leq_{\mathcal p} (W',\delta')$.
\end{proposition}

\begin{proof}
    If $(W',\delta')$ is itself a $\mathcal p$-minimal edge-weighted $r$-path vertex cover for $(\Sigma_{r-1}G)_\lambda$, then we are done. If $(W',\delta')$ is not $\mathcal p$-minimal, then by Lemma~\ref{carEqualityOfMinAndPMinLemmaW}, for some $v_i \in q(W')$ the function $\gamma_{(W',\delta')}(v_i) = \delta'(v_{i,j_0}) + \sum_{k=0}^{j_0-1}h_{i,k}$ with $j_0 := \{j \mid v_{i,j} \in W_i(\ffp)\}$ from Proposition~\ref{equivaDefForGammaVDeltaProp} can be increased, which is done by increasing $j_0$ and assigning an appropriate value to $\delta'(v_{i,j_0})$ since $(W',\delta')$ is minimal. We increase $\gamma_{(W',\delta')}(v_i)$ for each $v_i \in q(W')$ such that any further increase would cause the set not to be an edge-weighted $r$-path vertex cover. This process terminates in finitely many steps because $j_0 \leq r$. Denote the new set $(W'',\delta'')$. Then $(W'',\delta'')$ is minimal since the size of $W''$ cannot be decreased by Lemma~\ref{carEqualityOfMinAndPMinLemmaW} and $\delta''$ cannot be increased. Thus, by construction, $(W'',\delta'')$ is a $\mathcal p$-minimal edge-weighted $r$-path vertex cover for $(\Sigma_{r-1}G)_\lambda$ such that $(W'',\delta'') \leq_{\mathcal p} (W',\delta')$.
\end{proof}

The following example illustrates the previous proposition.

\begin{example}
    Consider the following minimal edge-weighted 3-path vertex cover $\ffp := (V_1'',\delta_1'') := \{v_{1,1}^5,v_2^3\}$ of $(\Sigma_2P_1)_\lambda$ as in Example~\ref{pMinWRPVCByCriterionAndComputePMinExample}\ref{pMinWRPVCByCriterionAndComputePMinExamplea}. 
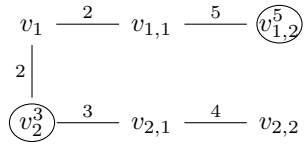
    \[
        \begin{tikzcd}[every matrix/.append style={name=m},
          execute at end picture={
              \draw [<-] (m-1-3) ellipse (0.3cm and 0.25cm);
              \draw [<-] (m-2-1) ellipse (0.3cm and 0.25cm);
          }]
            v_1 \ar[d,dash,"2"'] \ar[r,dash,"2"] & v_{1,1} \ar[r,dash,"5"] & v_{1,2}^5 \\
            v_2^3 \ar[r,dash,"3"] & v_{2,1} \ar[r,dash,"4"] & v_{2,2} 
        \end{tikzcd}
    \]
    Note that $\gamma_{(V_1'',\delta_1'')}(v_1)$ cannot be increased. Assume that $v_{2,1} \in V''$. Then set $\delta''(v_{2,1}) = 3$, we have that $\ffp' := (V_1''',\delta_1''') = \{v_{1,2}^5,v_{2,1}^3\}$ is a minimal edge-weighted 3-path vertex cover by Example~\ref{pMinWRPVCByCriterionAndComputePMinExample}\ref{pMinWRPVCByCriterionAndComputePMinExamplea}. However, since $v_{1,2} \in V''$, we have that $v_{2,2}$ cannot be in $V'''$, otherwise the 3-path $v_{1,1}v_1v_2v_{2,1}$ will be left uncovered. Thus, $(V_1''',\delta_1''')$ is $\mathcal p$-minimal and $(V_1''',\delta_1''') <_{\mathcal p} (V_1'',\delta_1'')$.
\end{example}

\begin{lemma} \label{MMinWRPathVCLemma}
    Let $(V_1',\delta_1'),(V_2',\delta_2')$ be two minimal edge-weighted $r$-path vertex covers of $(\Sigma_{r-1}G)_\lambda$. Then $(V_1',\delta_1') \leq_{\mathcal p} (V_2',\delta_2')$ if and only if $P(q(V_1'),\gamma_{(V_1',\delta_1')}) \subseteq P(q(V_2'),\gamma_{(V_2',\delta_2')})$.
\end{lemma}

\begin{proof}
    $(V_1',\delta_1') \leq_{\mathcal p} (V_2',\delta_2')$ if and only if $q(V_1') \subseteq q(V_2')$ and $\gamma_{(V_1',\delta_1')}|_{q(V_1')} \geq \gamma_{(V_2',\delta_2')}|_{q(V_1')}$ if and only if $P(q(V_1'),\gamma_{(V_1',\delta_1')}) \subseteq P(q(V_2'),\gamma_{(V_2',\delta_2')})$.
\end{proof}

The following example illustrates the previous lemma.

\begin{example}
    Consider the following two minimal edge-weighted 3-path vertex covers $(V_1'',\delta_1'') := \{v_{1,1}^5,v_2^3\}$ and $(V_2'',\delta_2'') := \{v_{1,2}^5,v_2^3\}$ of $(\Sigma_2P_1)_\lambda$ as in Example~\ref{pMinWRPVCByCriterionAndComputePMinExample}\ref{pMinWRPVCByCriterionAndComputePMinExamplea}. 
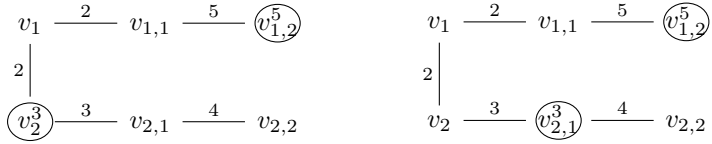
    \[
        \begin{tikzcd}[every matrix/.append style={name=m},
          execute at end picture={
              \draw [<-] (m-1-3) ellipse (0.3cm and 0.25cm);
              \draw [<-] (m-2-1) ellipse (0.3cm and 0.25cm);
          }]
            v_1 \ar[d,dash,"2"'] \ar[r,dash,"2"] & v_{1,1} \ar[r,dash,"5"] & v_{1,2}^5 \\
            v_2^3 \ar[r,dash,"3"] & v_{2,1} \ar[r,dash,"4"] & v_{2,2} 
        \end{tikzcd} \ \ \ \ \ \ \ \ \ \ \ \ 
        \begin{tikzcd}[every matrix/.append style={name=m},
          execute at end picture={
              \draw [<-] (m-1-3) ellipse (0.3cm and 0.25cm);
              \draw [<-] (m-2-2) ellipse (0.3cm and 0.25cm);
          }]
            v_1 \ar[d,dash,"2"'] \ar[r,dash,"2"] & v_{1,1} \ar[r,dash,"5"] & v_{1,2}^5 \\
            v_2 \ar[r,dash,"3"] & v_{2,1}^3 \ar[r,dash,"4"] & v_{2,2} 
        \end{tikzcd}
    \]
    Then $(V_2'',\delta_2'') <_{\mathcal p} (V_1'',\delta_1'')$ by Example~\ref{pMinWRPVCByCriterionAndComputePMinExample}\ref{pMinWRPVCByCriterionAndComputePMinExamplea}. Also note that 
    \[P(q(V_2''),\gamma_{(V_2'',\delta_2'')}) = (X_1^{12},X_2^6)R \subseteq (X_1^{12},X_2^3)R = P(q(V_1''),\gamma_{(V_1'',\delta_1'')}) .\]
\end{example}

Next, we present our third and final decomposition result which will yield the type computation in Theorem~\ref{computeTypeOfWRPathOfWRPathSuspension}.

\begin{theorem} \label{irredundantParaDecomOfWRPAthThm}
    Let $(\Sigma_rG)_\lambda$ be an edge-weighted $r$-path suspension of $G_\omega$ such that $\lambda(v_iv_j) \leq \lambda(v_i,v_{i,1})$ and $\lambda(v_iv_j) \leq \lambda(v_j,v_{j,1})$ for all edges $v_iv_j \in E$. One has an irredundant irreducible decomposition 
    \[I_r((\Sigma_rG)_\lambda)R = \bigcap_{(V'',\delta'') \text{ $\mathcal p$-min. w. $r$-path v. c. of }(\Sigma_{r-1}G)_{\lambda'}} P(q(V''),\gamma_{(V'',\delta'')}) + \ffm^{[\underline a(\lambda)]},\ \lambda' = \lambda|_{\Sigma_{r-1}G}.\]
\end{theorem}

\begin{proof}
    By Example~\ref{IrSrEIrSrM1PMExample} and \cite[Theorem 7.5.3]{MR3839602}, to verify this result, it is enough to show that we have an irredundant decomposition
    \[I_r((\Sigma_{r-1}G)_{\lambda'})R = \bigcap_{(V'',\delta'') \text{ $\mathcal p$-min. w. $r$-path v. cover of }(\Sigma_{r-1}G)_{\lambda'}} P(q(V''),\gamma_{(V'',\delta'')}).\] 
    Lemma~\ref{MMinWRPathVCLemma} shows that this intersection is irredundant. Hence by Corollary~\ref{minDecomOfReductiveWRPathIdealRMinus1}, it is enough to show that
    \begin{align*}
        &\ \ \ \ \bigcap_{(V'',\delta'') \text{ min. edge-weighted $r$-path v. cover of }(\Sigma_{r-1}G)_{\lambda'}} P(q(V''),\gamma_{(V'',\delta'')}) \\
        &= \bigcap_{(V'',\delta'') \text{ $\mathcal p$-min. edge-weighted $r$-path v. cover of }(\Sigma_{r-1}G)_{\lambda'}} P(q(V''),\gamma_{(V'',\delta'')}).
    \end{align*}
    \par ``$\subseteq$'' follows as every $\mathcal p$-minimal edge-weighted $r$-path vertex cover is a minimal edge-weighted $r$-path vertex cover. \par
    ``$\supseteq$'' follows from Proposition~\ref{existenceOfMMinWRPathVCOfAnyMProp} and Lemma~\ref{MMinWRPathVCLemma}.
\end{proof}

The following example illustrates the previous theorem.

\begin{example} \label{exampleOfDecomOfWPMRPRIdeal}
    Consider the graph $(\Sigma_3 P_1)_\lambda$ as in Example~\ref{exampleOfDecomOfWRPRIdeal}. 
    Then by Theorem~\ref{irredundantParaDecomOfWRPAthThm} and Example~\ref{pMinWRPVCByCriterionAndComputePMinExample}\ref{pMinWRPVCByCriterionAndComputePMinExampleb}, we have an irredundant irreducible decomposition
    \begin{align*}
        I_3((\Sigma_3P_1)_\lambda) &= (X_1^{12}X_2^2,X_1^4X_2^6,X_1^2X_2^{11})R + \ffm^{[\underline a(\lambda)]} \\
        &= \big[(X_1^2)R \cap (X_2^2)R \cap (X_1^{12},X_2^6)R \cap (X_1^4,X_2^{11})R ] + (X_1^{14},X_2^{13})R.
    \end{align*}
\end{example}

Here we provide two important facts which will be used in proving our main theorem.

\begin{fact} \label{moddingOutByRegularSequenceReplacingYWithXWRPathCase}
    Let $(\Sigma_rG)_\lambda$ be an edge-weighted $r$-path suspension of $G_\omega$ such that $\lambda(v_iv_j) \leq \lambda(v_i,v_{i,1})$ and $\lambda(v_iv_j) \leq \lambda(v_j,v_{j,1})$ for all edges $v_iv_j \in E$. Then $I_r((\Sigma_r G)_\lambda)$ is the polarization of $I_r((\Sigma_rG)_\lambda)R$ by e.g., \cite[Proposition 3.7]{MR3335988}. Hence by \cite{MR681256}, the list $X_i-X_{i,k},1 \leq i \leq d,1 \leq k \leq r$ is a maximal homogeneous regular sequence for $\frac{R'}{I_r((\Sigma_r G)_\lambda)}$ and
    \[\frac{R}{I_r((\Sigma_rG)_\lambda)R} \cong \frac{R'}{I_r((\Sigma_rG)_\lambda)+(X_i-X_{i,k} \mid 1 \leq i \leq d,1 \leq k \leq r)R'}.\]
\end{fact}

This fact is crucial in computing the Cohen-Macaulay type of an edge-weighted $r$-path suspension.

Because of the following fact, the main result of this section gives a formula to compute the $r_R(R/I_r(G_\omega))$ for all trees such that $R/I_r(G_\omega)$ is Cohen-Macaulay.

\begin{fact} \label{SigmaGIsCMWRPathCase} \cite[Proposition 3.7 and Theorem 3.11]{MR3335988}
    Let $(\Sigma_rG)_\lambda$ be an edge-weighted $r$-path suspension of $G_\omega$ such that $\lambda(v_iv_j) \leq \min\{\lambda(v_i,v_{i,1}),\lambda(v_j,v_{j,1})\}$ for all edges $v_iv_j \in E$. 
    \begin{enumerate}
        \item \label{SigmaGIsCMWRPathCasea} $R'/I_r((\Sigma_r G)_\lambda)$ is Cohen-Macaulay.
        \item \label{SigmaGIsCMWRPathCaseb} If $\Gamma_{\lambda'}$ is an edge-weighted tree and $R/I_r(\Gamma_{\lambda'})$ is Cohen-Macaulay, then there exists an edge-weighted tree $H_{\omega'}$ such that $(\Sigma_r H)_{\lambda''}$ is obtained by pruning a sequence of $r$-pathless leaves from $\Gamma_{\lambda'}$ with $\lambda'' = \lambda'|_{\Sigma_rH}$ and the weight function $\lambda'$ satisfies the above condition, where a vertex $v$ in $T_{\lambda'}$ is called an \emph{$r$-pathless leaf} of $T_{\lambda'}$  if it not a part of any $r$-path in $T_{\lambda'}$.
    \end{enumerate}
\end{fact} 

The following fact relates the Cohen-Macaulay type of $R/I$ to an irredundant irreducible decomposition of $I$ when $I$ is some special monomial ideal.

\begin{fact} \label{typeWhenMonomialIdealIrredundantParametricDecomp}
    Suppose that $I$ is a proper monomial ideal in $R$ such that $\dim(R/I) = 0$. Let $I = \bigcap_{i=1}^t Q_i$ be an irredundant irreducible decomposition of $I$. Then $r_R(R/I) = t$.
\end{fact}

    The next theorem is the main result of this paper. 

\begin{theorem} \label{computeTypeOfWRPathOfWRPathSuspension}
    Let $(\Sigma_rG)_\lambda$ be an edge-weighted $r$-path suspension of $G_\omega$ such that $\lambda(v_iv_j) \leq \lambda(v_i,v_{i,1})$ and $\lambda(v_iv_j) \leq \lambda(v_j,v_{j,1})$ for all edges $v_iv_j \in E$. Then
    \[r_{R'}\biggl(\frac{R'}{I_r((\Sigma_r G)_\lambda)}\biggr) =\sharp\+ \{\text{$\mathcal p$-minimal edge-weighted $r$-path vertex covers of $(\Sigma_{r-1}G)_{\lambda'}$}, \ \lambda' = \lambda|_{\Sigma_{r-1}G}\}.\]
\end{theorem}

\begin{proof}
    We compute
    \begin{align*}
        r_{R'}\biggl(\frac{R'}{I_r((\Sigma_rG)_\lambda)}\biggr) &= r_{R'}\biggl(\frac{R'}{I_r((\Sigma_r G)_\lambda)+(X_i-X_{i,k} \mid 1 \leq i \leq d,1 \leq k \leq r)R'}\biggr) \\
        &= r_R\biggl(\frac{R}{I_r((\Sigma_rG)_\lambda)R}\biggr) \\
        &=\sharp\+ \{\text{ideals in an irredundant irreducible decomposition of }I_r((\Sigma_rG)_\lambda)R\} \\
        &=\sharp\+ \{\text{$\mathcal p$-minimal edge-weighted $r$-path vertex covers of }(\Sigma_{r-1}G)_{\lambda'}\},
    \end{align*}
    where the first equality is from \cite[Lemma 1.3.16]{MR1251956}, \ref{SigmaGIsCMWRPathCase}\ref{SigmaGIsCMWRPathCasea} and Fact~\ref{moddingOutByRegularSequenceReplacingYWithXWRPathCase}, the second equality is from Fact~\ref{moddingOutByRegularSequenceReplacingYWithXWRPathCase}, the third equality is from Fact~\ref{typeWhenMonomialIdealIrredundantParametricDecomp} since $\dim(R/I_r((\Sigma_rG)_\lambda)R) = 0$, and the last equality is from Fact~\ref{irredundantParaDecomOfWRPAthThm}.
\end{proof}

The following example illustrates the previous theorem.

\begin{example}
    Consider Example~\ref{exampleOfDecomOfWPMRPRIdeal}. Then by Theorem~\ref{computeTypeOfWRPathOfWRPathSuspension}, we have that
    \[r_{R'}(R'/I_3(\Sigma_3P_1)_\lambda) = 4.\]
\end{example}

\begin{corollary}\label{computeTypeOfWAndUMRPathOfRPathSuspensionCor}
    We have the following.
    \begin{enumerate}
        \item \label{computeTypeOfWAndUMRPathOfRPathSuspensionCora} Let $(\Sigma_rG)_\lambda$ be an edge-weighted $r$-path suspension of $G_\omega$ such that $\lambda(v_iv_j) \leq \lambda(v_i,v_{i,1})$ and $\lambda(v_iv_j) \leq \lambda(v_j,v_{j,1})$ for all edges $v_iv_j \in E$. Then 
            \[r_{R'}\biggl(\frac{R'}{I_r(\Sigma_r G)}\biggr) =\sharp\+ \{\text{$\mathcal p$-minimal $r$-path vertex covers of $\Sigma_{r-1}G$}\}.\]
        \item \label{computeTypeOfWAndUMRPathOfRPathSuspensionCorb}
            Let $(\Sigma G)_\lambda$ be an edge-weighted suspension of $G_\omega$ such that $\lambda(v_iv_j) \leq \lambda(v_iw_i)$ and $\lambda(v_iv_j) \leq \lambda(w_jv_j)$ for each $v_iv_j \in E$. Then
            \[r_{R'}\biggl(\frac{R'}{I((\Sigma G)_\lambda)}\biggr) =\sharp\+ \{\text{minimal edge-weighted vertex covers of }G_\omega\}.\]
        \item \label{computeTypeOfWAndUMRPathOfRPathSuspensionCorc} Let $\Sigma G$ be a suspension of $G$. Then
            \[r_{R'}\biggl(\frac{R'}{I(\Sigma G)}\biggr) =\sharp\+ \{\text{minimal vertex covers of }G\}.\]
    \end{enumerate}
\end{corollary}

\begin{proof}
    \begin{enumerate}
        \item Let $\textbf 1: E(\Sigma_r G) \to \bbN$ be the constant weight function on $\Sigma_rG$ defined by $\textbf 1(e) = 1$ for $e \in E(\Sigma_rG)$. Then $\Sigma_{r-1}G = (\Sigma_{r-1}G)_{\textbf 1'}$ with $\textbf 1' = \textbf 1_{\lambda|_{\Sigma_{r-1}G}}$ and $I_r(\Sigma_rG) = I_r((\Sigma_rG)_{\textbf 1})$. Hence the conclusion is covered in Theorem~\ref{computeTypeOfWRPathOfWRPathSuspension}. 
        \item Let $(V'',\delta'')$ be a minimal $1$-path vertex cover of $G_\omega$. By definition, any $\mathcal p$-minimal $1$-path vertex cover of $G_\omega$ is a minimal $1$-path vertex cover. Then it suffices to show that $(V'',\delta'')$ is a $\mathcal p$-minimal $1$-path vertex cover of $G_\omega$, which is true by Theorem~\ref{criterionForPMinimalityThm}.
        \item Let $\textbf 1: E(\Sigma G) \to \bbN$ be the constant weight function on $\Sigma G$ defined by $\textbf 1(e) = 1$ for $e \in E(\Sigma G)$. Then $G = G_{\textbf 1}$ and $I(\Sigma G) = I((\Sigma G)_{\textbf 1})$. Therefore, the conclusion is covered in part (b). \qedhere
    \end{enumerate}
\end{proof}

Lastly, we give a corresponding example for each of previous corollaries.

\begin{example}
    \begin{enumerate}
        \item 
            Consider the following graph $\Sigma_3 P_2$ with $P_2 = (\begin{tikzcd}v_1 \ar[r,dash] & v_2 \ar[r,dash] & v_3\end{tikzcd})$
            \[
                \begin{tikzcd}
                    v_1 \ar[d,dash] \ar[r,dash] & v_{1,1} \ar[r,dash] & v_{1,2} \ar[r,dash] & v_{1,3} \\
                    v_2 \ar[d,dash] \ar[r,dash] & v_{2,1} \ar[r,dash] & v_{2,2} \ar[r,dash] & v_{2,3} \\
                    v_3 \ar[r,dash] & v_{3,1} \ar[r,dash] & v_{3,2} \ar[r,dash] & v_{3,3}
                \end{tikzcd}
            \]
            We depict the minimal 3-path vertex covers of $\Sigma_2P_2$ in the following sketches.
            
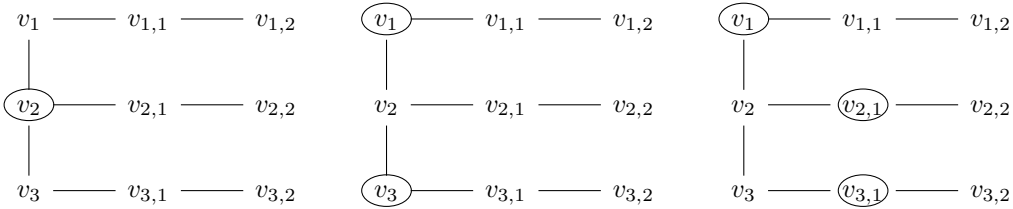
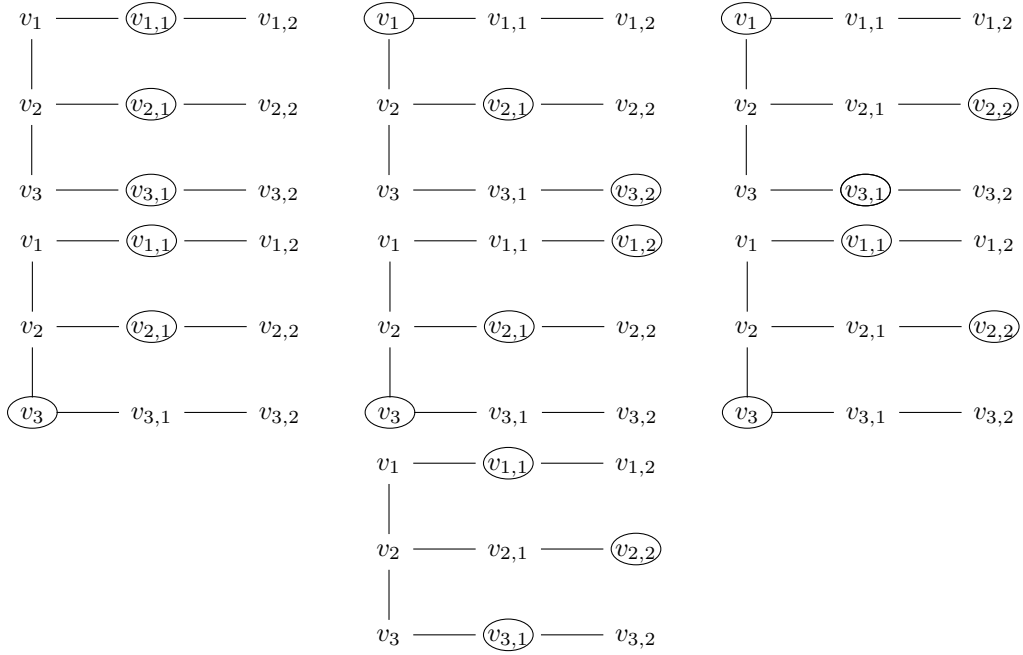
\[
                \begin{tikzcd}[every matrix/.append style={name=m},
                  execute at end picture={
                      \draw [<-] (m-2-1) ellipse (0.33cm and 0.21cm);
                  }]
                    v_1 \ar[d,dash] \ar[r,dash] & v_{1,1} \ar[r,dash] & v_{1,2} \\
                    v_2 \ar[d,dash] \ar[r,dash] & v_{2,1} \ar[r,dash] & v_{2,2} \\
                    v_3 \ar[r,dash] & v_{3,1} \ar[r,dash] & v_{3,2}
                \end{tikzcd} \ \ \ \ \ \
                \begin{tikzcd}[every matrix/.append style={name=m},
                  execute at end picture={
                      \draw [<-] (m-1-1) ellipse (0.33cm and 0.21cm);
                      \draw [<-] (m-3-1) ellipse (0.33cm and 0.21cm);
                  }]
                    v_1 \ar[d,dash] \ar[r,dash] & v_{1,1} \ar[r,dash] & v_{1,2} \\
                    v_2 \ar[d,dash] \ar[r,dash] & v_{2,1} \ar[r,dash] & v_{2,2} \\
                    v_3 \ar[r,dash] & v_{3,1} \ar[r,dash] & v_{3,2}
                \end{tikzcd} \ \ \ \ \ \
                \begin{tikzcd}[every matrix/.append style={name=m},
                  execute at end picture={
                      \draw [<-] (m-1-1) ellipse (0.33cm and 0.21cm);
                      \draw [<-] (m-2-2) ellipse (0.33cm and 0.21cm);
                      \draw [<-] (m-3-2) ellipse (0.33cm and 0.21cm);
                  }]
                    v_1 \ar[d,dash] \ar[r,dash] & v_{1,1} \ar[r,dash] & v_{1,2} \\
                    v_2 \ar[d,dash] \ar[r,dash] & v_{2,1} \ar[r,dash] & v_{2,2} \\
                    v_3 \ar[r,dash] & v_{3,1} \ar[r,dash] & v_{3,2}
                \end{tikzcd}
            \]
            \[
                \begin{tikzcd}[every matrix/.append style={name=m},
                  execute at end picture={
                      \draw [<-] (m-1-2) ellipse (0.33cm and 0.21cm);
                      \draw [<-] (m-2-2) ellipse (0.33cm and 0.21cm);
                      \draw [<-] (m-3-2) ellipse (0.33cm and 0.21cm);
                  }]
                    v_1 \ar[d,dash] \ar[r,dash] & v_{1,1} \ar[r,dash] & v_{1,2} \\
                    v_2 \ar[d,dash] \ar[r,dash] & v_{2,1} \ar[r,dash] & v_{2,2} \\
                    v_3 \ar[r,dash] & v_{3,1} \ar[r,dash] & v_{3,2}
                \end{tikzcd} \ \ \ \ \ \
                \begin{tikzcd}[every matrix/.append style={name=m},
                  execute at end picture={
                      \draw [<-] (m-1-1) ellipse (0.33cm and 0.21cm);
                      \draw [<-] (m-2-2) ellipse (0.33cm and 0.21cm);
                      \draw [<-] (m-3-3) ellipse (0.33cm and 0.21cm);
                  }]
                    v_1 \ar[d,dash] \ar[r,dash] & v_{1,1} \ar[r,dash] & v_{1,2} \\
                    v_2 \ar[d,dash] \ar[r,dash] & v_{2,1} \ar[r,dash] & v_{2,2} \\
                    v_3 \ar[r,dash] & v_{3,1} \ar[r,dash] & v_{3,2}
                \end{tikzcd} \ \ \ \ \ \
                \begin{tikzcd}[every matrix/.append style={name=m},
                  execute at end picture={
                      \draw [<-] (m-1-1) ellipse (0.33cm and 0.21cm);
                      \draw [<-] (m-2-3) ellipse (0.33cm and 0.21cm);
                      \draw [<-] (m-3-2) ellipse (0.33cm and 0.21cm);
                      \draw [<-] (m-3-2) ellipse (0.33cm and 0.21cm);
                  }]
                    v_1 \ar[d,dash] \ar[r,dash] & v_{1,1} \ar[r,dash] & v_{1,2} \\
                    v_2 \ar[d,dash] \ar[r,dash] & v_{2,1} \ar[r,dash] & v_{2,2} \\
                    v_3 \ar[r,dash] & v_{3,1} \ar[r,dash] & v_{3,2}
                \end{tikzcd}
            \]
            \[
                \begin{tikzcd}[every matrix/.append style={name=m},
                  execute at end picture={
                      \draw [<-] (m-1-2) ellipse (0.33cm and 0.21cm);
                      \draw [<-] (m-2-2) ellipse (0.33cm and 0.21cm);
                      \draw [<-] (m-3-1) ellipse (0.33cm and 0.21cm);
                  }]
                    v_1 \ar[d,dash] \ar[r,dash] & v_{1,1} \ar[r,dash] & v_{1,2} \\
                    v_2 \ar[d,dash] \ar[r,dash] & v_{2,1} \ar[r,dash] & v_{2,2} \\
                    v_3 \ar[r,dash] & v_{3,1} \ar[r,dash] & v_{3,2}
                \end{tikzcd} \ \ \ \ \ \
                \begin{tikzcd}[every matrix/.append style={name=m},
                  execute at end picture={
                      \draw [<-] (m-1-3) ellipse (0.33cm and 0.21cm);
                      \draw [<-] (m-2-2) ellipse (0.33cm and 0.21cm);
                      \draw [<-] (m-3-1) ellipse (0.33cm and 0.21cm);
                  }]
                    v_1 \ar[d,dash] \ar[r,dash] & v_{1,1} \ar[r,dash] & v_{1,2} \\
                    v_2 \ar[d,dash] \ar[r,dash] & v_{2,1} \ar[r,dash] & v_{2,2} \\
                    v_3 \ar[r,dash] & v_{3,1} \ar[r,dash] & v_{3,2}
                \end{tikzcd} \ \ \ \ \ \
                \begin{tikzcd}[every matrix/.append style={name=m},
                  execute at end picture={
                      \draw [<-] (m-1-2) ellipse (0.33cm and 0.21cm);
                      \draw [<-] (m-2-3) ellipse (0.33cm and 0.21cm);
                      \draw [<-] (m-3-1) ellipse (0.33cm and 0.21cm);
                  }]
                    v_1 \ar[d,dash] \ar[r,dash] & v_{1,1} \ar[r,dash] & v_{1,2} \\
                    v_2 \ar[d,dash] \ar[r,dash] & v_{2,1} \ar[r,dash] & v_{2,2} \\
                    v_3 \ar[r,dash] & v_{3,1} \ar[r,dash] & v_{3,2}
                \end{tikzcd}
            \]
            \[
                \begin{tikzcd}[every matrix/.append style={name=m},
                  execute at end picture={
                      \draw [<-] (m-1-2) ellipse (0.33cm and 0.21cm);
                      \draw [<-] (m-2-3) ellipse (0.33cm and 0.21cm);
                      \draw [<-] (m-3-2) ellipse (0.33cm and 0.21cm);
                  }]
                    v_1 \ar[d,dash] \ar[r,dash] & v_{1,1} \ar[r,dash] & v_{1,2} \\
                    v_2 \ar[d,dash] \ar[r,dash] & v_{2,1} \ar[r,dash] & v_{2,2} \\
                    v_3 \ar[r,dash] & v_{3,1} \ar[r,dash] & v_{3,2}
                \end{tikzcd}
            \]
            Hence by Corollary~\ref{minDecomOfReductiveWRPathIdealRMinus1}, we have an irreducible decomposition 
            \begin{align*}
                I_3(\Sigma_2P_2)R &= (X_2)R \cap  (X_1,X_3)R \cap (X_1,X_2^2,X_3^2)R \cap (X_1^2,X_2^2,X_3^2)R \cap (X_1,X_2^2,X_3^3)R \\
                &\ \ \ \cap (X_1,X_2^3,X_3^2)R \cap (X_1^2,X_2^2,X_3)R \cap (X_1^3,X_2^2,X_3)R \cap (X_1^2,X_2^3,X_3)R \\
                & \ \ \ \cap (X_1^2,X_2^3,X_3^2)R,
            \end{align*}
            which is a redundant decomposition since e.g., the last ideal $(X_1^2,X_2^3,X_3^2)R$ is contained in the second to last ideal $(X_1^2,X_2^3,X_3)R$. Note that the $\mathcal p$-minimal 3-path vertex covers of $\Sigma_2 P_2$ are the following.
            
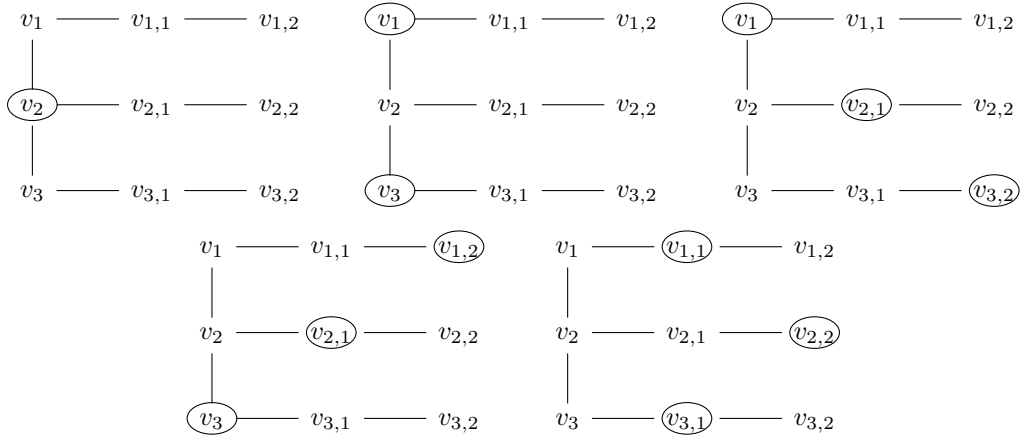
\[
                \begin{tikzcd}[every matrix/.append style={name=m},
                  execute at end picture={
                      \draw [<-] (m-2-1) ellipse (0.33cm and 0.21cm);
                  }]
                    v_1 \ar[d,dash] \ar[r,dash] & v_{1,1} \ar[r,dash] & v_{1,2} \\
                    v_2 \ar[d,dash] \ar[r,dash] & v_{2,1} \ar[r,dash] & v_{2,2} \\
                    v_3 \ar[r,dash] & v_{3,1} \ar[r,dash] & v_{3,2}
                \end{tikzcd} \ \ \ \ \ \
                \begin{tikzcd}[every matrix/.append style={name=m},
                  execute at end picture={
                      \draw [<-] (m-1-1) ellipse (0.33cm and 0.21cm);
                      \draw [<-] (m-3-1) ellipse (0.33cm and 0.21cm);
                  }]
                    v_1 \ar[d,dash] \ar[r,dash] & v_{1,1} \ar[r,dash] & v_{1,2} \\
                    v_2 \ar[d,dash] \ar[r,dash] & v_{2,1} \ar[r,dash] & v_{2,2} \\
                    v_3 \ar[r,dash] & v_{3,1} \ar[r,dash] & v_{3,2}
                \end{tikzcd} \ \ \ \ \ \
                \begin{tikzcd}[every matrix/.append style={name=m},
                  execute at end picture={
                      \draw [<-] (m-1-1) ellipse (0.33cm and 0.21cm);
                      \draw [<-] (m-2-2) ellipse (0.33cm and 0.21cm);
                      \draw [<-] (m-3-3) ellipse (0.33cm and 0.21cm);
                  }]
                    v_1 \ar[d,dash] \ar[r,dash] & v_{1,1} \ar[r,dash] & v_{1,2} \\
                    v_2 \ar[d,dash] \ar[r,dash] & v_{2,1} \ar[r,dash] & v_{2,2} \\
                    v_3 \ar[r,dash] & v_{3,1} \ar[r,dash] & v_{3,2}
                \end{tikzcd}
            \]
            \[
                \begin{tikzcd}[every matrix/.append style={name=m},
                  execute at end picture={
                      \draw [<-] (m-1-3) ellipse (0.33cm and 0.21cm);
                      \draw [<-] (m-2-2) ellipse (0.33cm and 0.21cm);
                      \draw [<-] (m-3-1) ellipse (0.33cm and 0.21cm);
                  }]
                    v_1 \ar[d,dash] \ar[r,dash] & v_{1,1} \ar[r,dash] & v_{1,2} \\
                    v_2 \ar[d,dash] \ar[r,dash] & v_{2,1} \ar[r,dash] & v_{2,2} \\
                    v_3 \ar[r,dash] & v_{3,1} \ar[r,dash] & v_{3,2}
                \end{tikzcd} \ \ \ \ \ \
                \begin{tikzcd}[every matrix/.append style={name=m},
                  execute at end picture={
                      \draw [<-] (m-1-2) ellipse (0.33cm and 0.21cm);
                      \draw [<-] (m-2-3) ellipse (0.33cm and 0.21cm);
                      \draw [<-] (m-3-2) ellipse (0.33cm and 0.21cm);
                  }]
                    v_1 \ar[d,dash] \ar[r,dash] & v_{1,1} \ar[r,dash] & v_{1,2} \\
                    v_2 \ar[d,dash] \ar[r,dash] & v_{2,1} \ar[r,dash] & v_{2,2} \\
                    v_3 \ar[r,dash] & v_{3,1} \ar[r,dash] & v_{3,2}
                \end{tikzcd}
            \]
            Then by Theorem~\ref{irredundantParaDecomOfWRPAthThm} and we have an irredundant irreducible decomposition
            \begin{align*}
                I_3(\Sigma_2P_2)R &= (X_2)R \cap  (X_1,X_3)R \cap (X_1,X_2^2,X_3^3)R \cap (X_1^3,X_2^2,X_3)R \cap (X_1^2,X_2^3,X_3^2)R,
            \end{align*}
            and by Corollary~\ref{computeTypeOfWAndUMRPathOfRPathSuspensionCor}\ref{computeTypeOfWAndUMRPathOfRPathSuspensionCora}, we have that
            \[r_{R'}(R'/I_3(\Sigma_3P_2)) = 5.\]
        \item 
            Consider the following edge-weighted graph $(\Sigma P_2)_\lambda$ with $(P_2)_\omega = (\begin{tikzcd}v_1 \ar[r,dash,"2"] & v_2 \ar[r,dash,"3"] & v_3\end{tikzcd})$.
            \[
                \begin{tikzcd}
                    w_1 \ar[d,dash,"5"] & w_2 \ar[d,dash,"3"] & w_3 \ar[d,dash,"4"] \\
                    v_1 \ar[r,dash,"2"] & v_2 \ar[r,dash,"3"] & v_3
                \end{tikzcd}
            \]
            The minimal edge-weighted vertex covers of $(P_2)_\omega = \Bigl(\begin{tikzcd}v_1 \ar[r,dash,"2"] & v_2 \ar[r,dash,"3"] & v_3 \end{tikzcd}\Bigr)$ are displayed in the following sketches.
            \[
                \begin{tikzcd}[every matrix/.append style={name=m},
                  execute at end picture={
                      \draw [<-] (m-1-2) ellipse(0.3cm and 0.25cm);
                  }]
                  v_1 \ar[r,dash,"2"] & v_2^2 \ar[r,dash,"3"] & v_3
                \end{tikzcd} \ \ \ \ \ \ 
                \begin{tikzcd}[every matrix/.append style={name=m},
                  execute at end picture={
                      \draw [<-] (m-1-1) ellipse(0.3cm and 0.25cm);
                      \draw [<-] (m-1-2) ellipse(0.3cm and 0.25cm);
                  }]
                  v_1^2 \ar[r,dash,"2"] & v_2^3 \ar[r,dash,"3"] & v_3
                \end{tikzcd} \ \ \ \ \ \ 
                \begin{tikzcd}[every matrix/.append style={name=m},
                  execute at end picture={
                      \draw [<-] (m-1-1) ellipse(0.3cm and 0.25cm);
                      \draw [<-] (m-1-3) ellipse(0.3cm and 0.25cm);
                  }]
                  v_1^2 \ar[r,dash,"2"] & v_2 \ar[r,dash,"3"] & v_3^3
                \end{tikzcd}
            \]
            Then by Corollary~\ref{computeTypeOfWAndUMRPathOfRPathSuspensionCor}\ref{computeTypeOfWAndUMRPathOfRPathSuspensionCorb},
            \[r_{R'}(R'/I((\Sigma P_2)_\lambda)) =\sharp\+ \{\text{minimal edge-weighted vertex covers of }(P_2)_\omega = 3\}.\]
        \item 
            Consider the following graph $\Sigma P_2$ with $P_2 = (\begin{tikzcd}v_1 \ar[r,dash] & v_2 \ar[r,dash] & v_3\end{tikzcd})$.
            \[
                \begin{tikzcd}
                    w_1 \ar[d,dash] & w_2 \ar[d,dash] & w_3 \ar[d,dash] \\
                    v_1 \ar[r,dash] & v_2 \ar[r,dash] & v_3
                \end{tikzcd}
            \]
            We depict the minimal vertex covers of $P_2$ in the following sketches.
            \[
                \begin{tikzcd}[every matrix/.append style={name=m},
                  execute at end picture={
                      \draw [<-] (m-1-1) ellipse (0.3cm and 0.2cm);
                      \draw [<-] (m-1-3) ellipse (0.3cm and 0.2cm);
                  }]
                    v_1 \ar[r,dash] & v_2 \ar[r,dash] & v_3
                \end{tikzcd} \ \ \ \ \ \ 
                \begin{tikzcd}[every matrix/.append style={name=m},
                  execute at end picture={
                      \draw [<-] (m-1-2) ellipse (0.3cm and 0.2cm);
                  }]
                    v_1 \ar[r,dash] & v_2 \ar[r,dash] & v_3
                \end{tikzcd}
            \]
            By Corollary~\ref{computeTypeOfWAndUMRPathOfRPathSuspensionCor}\ref{computeTypeOfWAndUMRPathOfRPathSuspensionCorc}, 
            \[r_{R'}(R'/I(\Sigma P_2)) =\sharp\+ \{\text{minimal vertex covers of }P_2\} = 2.\]
    \end{enumerate}
\end{example}

\section*{Acknowledgments}

We are grateful for the insightful comments and feedback provided by Keri Ann Sather-Wagstaff and Janet Vassilev. 

\bibliographystyle{plain} 
\bibliography{bibliography}

\end{document}